\documentclass{article}
\usepackage[utf8]{inputenc}
\usepackage[T2A]{fontenc}
\usepackage{fullpage}
\usepackage{amsmath,amsfonts,amssymb,amsthm,enumerate,mdframed,hyperref,cleveref,algorithmic,algorithm,multirow,import,booktabs,xcolor,soul}
\usepackage{stmaryrd}
\usepackage{caption,subcaption}
\usepackage{tikz}  
\usetikzlibrary{arrows,shapes,trees,positioning,quotes}
\newcommand{\tikzmark}[1]{\tikz[baseline,remember picture] \coordinate (#1) {};}

\newtheorem{proposition}{Proposition}[section] % [subsection]?
\newtheorem{lemma}[proposition]{Lemma}
\newtheorem{theorem}[proposition]{Theorem}
\newtheorem{conjecture}{Conjecture}
\newtheorem{corollary}[proposition]{Corollary}
\theoremstyle{definition}
\newtheorem{definition}[proposition]{Definition}

\providecommand{\lmarked}[1]{\color{blue}{#1}}
\providecommand{\nones}[1]{\#_1(#1)}
\providecommand{\pdot}{\mathord{\cdot}}
\providecommand{\ZZf}{\mathbb{Z}_4}
\providecommand{\WWf}{\mathbb{W}_4}
\providecommand{\enc}[1]{\llbracket{#1}\rrbracket}
\providecommand{\taby}{\checkmark}
\providecommand{\tabn}{\times}

\mathcode`*=\numexpr\mathcode`*-"2000\relax

\title{Irreducible Subcube Partitions}
\author{Yuval Filmus, Edward A. Hirsch, Sascha Kurz, Ferdinand Ihringer, \\ Artur Riazanov, Alexander Smal, Marc Vinyals}

\begin{document}

\maketitle

\begin{abstract}
A \emph{subcube partition} is a partition of the Boolean cube $\{0,1\}^n$ into subcubes. A subcube partition is \emph{irreducible} if the only sub-partitions whose union is a subcube are singletons and the entire partition. A subcube partition is \emph{tight} if it ``mentions'' all coordinates.

We study extremal properties of tight irreducible subcube partitions: minimal size, minimal weight, maximal number of points, maximal size, and maximal minimum dimension.
We also consider the existence of \emph{homogeneous} tight irreducible subcube partitions, in which all subcubes have the same dimensions. We additionally study subcube partitions of $\{0,\dots,q-1\}^n$, and partitions of $\mathbb{F}_2^n$ into affine subspaces, in both cases focusing on the minimal size.

Our constructions and computer experiments lead to several conjectures on the extremal values of the aforementioned properties.
\end{abstract}

\section{Introduction}
\label{sec:introduction}

A \emph{subcube partition} is a partition of the cube $\{0,1\}^n$ into subcubes, that is, into sets of the form
\[
 \{ x \in \{0,1\}^n : x_{i_1} = b_1, \ldots, x_{i_d} = b_d \}.
\]
Here is an example of a subcube partition of \emph{length} $n = 3$:
\begin{align*}
 S_3 &= \{000\}, \{111\}, \{001, 101\}, \{100, 110\}, \{010, 011\} \\ &=
 000, 111, *01, 1*0, 01*.
\end{align*}
We will usually express our subcubes as strings in $\{0,1,*\}^n$, in which stars stand for unconstrained coordinates.

A subcube partition is \emph{reducible} if it has a proper subset, consisting of more than one subcube, whose union is a subcube. For example,
\[
 0*, 10, 11
\]
is reducible since $10 \cup 11 = 1*$. In contrast, $S_3$ is \emph{irreducible}.

A subcube partition is \emph{tight} if it \emph{mentions} all coordinates, that is, if for every $i \in [n]$, some subcube constrains $x_i$. Both subcube partitions above are tight, but the subcube partition $0*,1*$ is not, since the second coordinate is not mentioned.

Peitl and Szeider~\cite{PeitlSzeider22} enumerated all tight irreducible subcube partitions for $n = 3,4$, and counted the number of nonisomorphic subcube partitions with small \emph{size} (number of subcubes) for $n=5,6,7$. They ask whether there are infinitely many tight irreducible subcube partitions.
In this work, we answer this question in the affirmative, giving many constructions of tight irreducible subcube partitions.

The work of Peitl and Szeider raises many natural questions, such as:
\begin{itemize}
\item How to determine whether a subcube partition is irreducible?
\item What is the minimal size of a tight irreducible subcube partition of length $n$? \\ (This question only makes sense if we impose tightness.)
\item What is the maximal size of an irreducible subcube partition of length $n$?
\item Do there exist irreducible subcube partitions in which all subcubes have the same dimension? \\ (We call such subcube partitions \emph{homogeneous}.)
\end{itemize}

We address these questions in \Cref{sec:subcube-partitions}. We describe an efficient algorithm for testing whether a subcube partition is irreducible in \Cref{sec:testing-irreducibility}, and give two infinite sequences of irreducible formulas in \Cref{sec:kis-construction}.

We conjecture that the minimal size of a tight irreducible subcube partition of length $n$ is $2n-1$. We give a matching construction in \Cref{sec:minimal-subcubes}, and optimize its Hamming weight in \Cref{sec:minimal-weight} (this will be useful later on).

We give in \Cref{sec:maximal-subcubes} a construction of an irreducible subcube partition consisting of $2^{n-2}$ points ($0$-dimensional subcubes) and $3 \cdot 2^{n-3}$ edges ($1$-dimensional subcubes), for any $n \ge 5$. A computer check shows that this construction has the maximal size for $n = 5$ and $n = 6$, but not for $n = 7$. We address the related question of the maximal number of points in an irreducible subcube partition in \Cref{sec:maximal-points}.

% We conjecture that the maximal size of an irreducible subcube partition of length $n \ge 5$ is $\tfrac58 2^n$. We give a matching construction in \Cref{sec:maximal-subcubes}, where we also give a nontrivial upper bound. Our constructions involve $2^{n-2}$ points ($0$-dimensional subcubes) and $3 \cdot 2^{n-3}$ edges ($1$-dimensional subcubes). We conjecture that $2^{n-2}$ is the maximum number of points in an irreducible subcube partition of size $n$. A simple matching construction appears in \Cref{sec:maximal-points}.

We present subcube partitions in which all subcubes have linear dimension in \Cref{sec:maximal-min-dim}. We close off the section with a discussion of homogeneous subcube partitions in \Cref{sec:homogeneous-partitions}, where we describe several constructions, and determine all lengths $n$ and codimensions $k \le 4$ for which there exists a tight irreducible subcube partition of length $n$ whose subcubes have dimension $n - k$. In particular, we describe a construction due to Perezhogin~\cite{Perezhogin05} of irreducible subcube partitions of all length $n \ge 7$ in which all subcubes have dimension~$1$.

\smallskip

\Cref{sec:larger-alphabets} studies subcube partitions of $[q]^n$ for $q \ge 3$. We show how to construct irreducible subcube partitions of $[q]^n$ from irreducible subcube partitions of $\{0,1\}^n$ in \Cref{sec:expansion}, and use this to construct tight irreducible subcube partitions of length $n$ and size $(n-1)q(q-1) + 1$ in \Cref{sec:larger-alphabets-construction}; this uses the subcube partitions of \Cref{sec:minimal-weight}. We conjecture that $(n-1)q(q-1)+1$ is the minimum size of a tight irreducible subcube partition, and prove this for $n = 3$ in \Cref{sec:larger-alphabets-small-n}. We close by showing, in \Cref{sec:covers}, that the minimal size of a tight \emph{minimal cover} in this setting is $(q-1)n + 1$.

\smallskip

Finally, \Cref{sec:affine-vector-space-partitions} briefly studies the linear analog of subcube partitions, in which subcubes are replaced by affine subspaces. We show how to construct irreducible affine vector space partitions from irreducible subcube partitions in \Cref{sec:compression}, and use this to construct tight irreducible affine subspace partitions of length $n$ and size roughly $\tfrac32n$ in \Cref{sec:avsps-minimal-subcubes}. We discuss irreducible affine vector space partitions in more detail in the companion work~\cite{BFIK22}.

\subsection*{Background}
\label{sec:background}

Subcube partitions appear, under various names, in theoretical computer science, as an abstraction of the salient properties of decision trees, and elsewhere. Some examples include Iwama~\cite{Iwama87,Iwama89} (as certain \emph{independent sets of clauses}), Brandman, Orlitsky and Hennessy~\cite{BOH90} (as \emph{nonoverlapping covers}), Berger, Felzanbaum and Fraenkel~\cite{BFF90} (as \emph{disjoint tautologies}), Davydov and Davydova~\cite{DavDav98} (as \emph{dividing formulas}), Friedgut, Kahn and Wigderson~\cite{FKW02} (as \emph{subcube partitions}), Kullmann~\cite{Kullmann04} (as \emph{unsatisfiable hitting clause-sets}), Kisielewicz~\cite{Kis20} (as realizations of \emph{cube tiling codes}), Tarannikov~\cite{Tarannikov22} (as \emph{coordinate partitions}). There are also \emph{orthogonal DNFs}~\cite{CH11}, also known as \emph{disjoint DNFs}~\cite{GwynneKullman13}, which are systems of disjoint subcubes which do not necessarily cover the entire cube. (For the relation between decision trees and subcube partitions, see G\"o\"os, Pitassi and Watson~\cite{GPW18}.)

Irreducible subcube partitions appear in a work of Kullmann and Zhao~\cite{KullmannZhao16} (as \emph{clause-reducibility}), inspired by similar notions in the context of disjoint covering systems of residue classes~\cite{Korec84,BFF90} and motivated by applications to the study of CNFs.

Peitl and Szeider~\cite{PeitlSzeider22} enumerate all tight irreducible subcube partitions for $n = 3,4$, and determine the minimal size of a regular irreducible subcube partition for $n=5,6,7$. Instead of tightness, they use a different notion, \emph{regularity}, which is equivalent to tightness for irreducible subcube partitions when $n \ge 3$. Regularity was introduced by Kullmann and Zhao~\cite{KullmannZhao13} under the name \emph{nonsingularity}, and is defined in \Cref{sec:size-lb}.

\paragraph{Acknowledgements}  We thank Jan Johannsen, Ilario Bonacina, Oliver Kullmann, and Stefan Szeider for introducing us to the topic; Tomáš Peitl and Stefan Szeider for sharing with us the results of the computer search reported in~\cite{PeitlSzeider22}; Zachary Chase for helpful discussions; Yuriy Tarannikov for bringing into our attention the papers~\cite{Perezhogin05,Agievich08,Tarannikov22} and helping us to understand~\cite{Perezhogin05}; Andrzej Kisielewicz for simplifying several of our proofs and for bringing into our attention~\cite{Kisielewicz14,KP08}; and the anonymous reviewer for their careful reading of the manuscript and helpful comments.

This project has received funding from the European Union's Horizon 2020 research and innovation programme under grant agreement No~802020-ERC-HARMONIC.
Ferdinand Ihringer is supported by a postdoctoral fellowship of the Research Foundation -- Flanders (FWO).

Since the original version of this paper was published, we ran further experiments which have refuted two of our conjectures (on the maximal size and on the maximal number of points). The experiments were orchestrated by GPT-6 Astra. We have updated the paper accordingly.

\section{Subcube partitions}
\label{sec:subcube-partitions}

We start with a quick recap of the relevant definitions.

\begin{definition}[Subcube partition]
\label{def:subcube-partition}
A \emph{subcube partition} of length $n$ is a partition of $\{0,1\}^n$ into \emph{subcubes}, which are sets of the form
\[
 \{ x \in \{0,1\}^n : x_{i_1} = b_1, \ldots, x_{i_d} = b_d \}.
\]
The parameter $d$ is the \emph{codimension} of the subcube, and $n-d$ is its \emph{dimension}. A subcube of dimension $0$ is called a \emph{point}, and a subcube of dimension $1$ is called an \emph{edge}.

The \emph{size} of a subcube partition is the number of subcubes.
\end{definition}

We identify subcubes with words over $\{0,1,*\}$. For example, $01*$ stands for the subcube $\{ (0,1,0), (0,1,1) \}$. We index the symbols in a word $w$ of length $n$ by $[n] = \{1,\dots,n\}$. If $b \in \{0,1\}$, we use $\bar{b}$ to denote $1-b$.

\begin{definition}[Reducibility]
\label{def:reducible}
A subcube partition $F$ is \emph{reducible} if there exists a subset $G \subset F$, with $1 < |G| < |F|$, such that the union of the subcubes in $G$ is itself a subcube.

A subcube partition is \emph{irreducible} if it is not reducible.
\end{definition}

% effective dimension?

\begin{definition}[Tightness]
\label{def:tight}
A subcube $s$ \emph{mentions} a coordinate $i \in [n]$ if $s_i \neq *$.

A subcube partition $F$ of length $n$ is \emph{tight} if for every $i \in [n]$, some subcube in $F$ mentions $i$.
\end{definition}

It is coNP-complete to determine whether a given collection of subcubes covers $\{0,1\}^n$ (this problem is just SAT in disguise). In contrast, it is easy to test whether a given collection of subcubes is a partition, as first observed by Iwama~\cite{Iwama89}.

\begin{definition}[Conflicting subcubes]
\label{def:conflicting-subcubes}
Two subcubes $s,t$ of the same length are said to \emph{conflict} if there is a coordinate $i \in [n]$ such that $s_i,t_i \neq *$ and $s_i \neq t_i$.
%$(s_i,t_i) \in \{(0,1),(1,0)\}$.
\end{definition}

\begin{lemma} \label{lem:disjoint-subcubes}
Two subcubes are disjoint if and only if they conflict.
\end{lemma}

\begin{lemma} \label{lem:testing-partition}
A collection $F$ of disjoint subcubes of length $n$ is a subcube partition if and only if
\[
 \sum_{s \in F} 2^{-\operatorname{codim}(s)} = 1.
\]
\end{lemma}

Similarly, it is easy to check whether a given subcube partition is tight. In contrast, checking whether a subcube partition is irreducible using the definition takes exponential time. We present an efficient algorithm for testing irreducibility in \Cref{sec:testing-irreducibility}.

Following that, we give many examples of irreducible subcube partitions, starting with \Cref{sec:kis-construction}, which describes subcube partitions whose irreducibility can be proved using a simple parity argument. In \Cref{sec:minimal-subcubes,sec:minimal-weight,sec:maximal-points,sec:maximal-subcubes,sec:maximal-min-dim} we describe irreducible subcube partitions, some of which conjecturally optimize various parameters. \Cref{sec:homogeneous-partitions} closes with a discussion of irreducible subcube partitions in which all subcubes have the same dimension.

\subsection{Testing irreducibility}
\label{sec:testing-irreducibility}

In this section we give a polynomial time algorithm that checks whether a given subcube partition $F$ is reducible, and if so, identifies a subset $G \subset F$, with $1 < |G| < |F|$, whose union is a subcube.

The idea behind the algorithm is quite simple. Suppose that $F$ were reducible, say via the subset $G$. If $s,t \in G$ then $\bigcup G$ must contain the \emph{join} $s \lor t$ of $s,t$, which is the smallest subcube containing both $s$ and $t$, given explicitly by
\[
 (s \lor t)_i =
 \begin{cases}
 b & \text{if } s_i = t_i = b \in \{0, 1\}, \\
 * & \text{otherwise}.
 \end{cases}
\]
If $u \in F$ intersects $s \lor t$ (a condition we can check using \Cref{lem:disjoint-subcubes}) then $G$ must contain $u$, and so $\bigcup G$ must contain $s \lor t \lor u$.
Continuing in this way, we are able to recover $G$ (or a subset of $G$ whose union is also a subcube).
The corresponding algorithm appears as \Cref{alg:reducibility}.

\renewcommand{\algorithmicrequire}{\textbf{Input:}}
\renewcommand{\algorithmicensure}{\textbf{Output:}}
\begin{algorithm} 
\begin{algorithmic}
\REQUIRE Subcube partition $F = \{s_1, \dots, s_m\}$
\FOR{$1 \leq i < j \leq m$}
    \STATE $G \gets \{s_i, s_j\}$
    \WHILE{some $s_k \notin G$ intersects $\bigvee G$}
        \STATE{$G \gets G \cup \{s_k\}$}
    \ENDWHILE
    \IF{$G \neq F$}
        \RETURN Reducible: $\bigcup G$ is a subcube
    \ENDIF
\ENDFOR
\RETURN Irreducible
\end{algorithmic}
\caption{Algorithm for checking whether a subcube partition is irreducible}
\label{alg:reducibility}
\end{algorithm}

\begin{theorem} \label{thm:reducibility-algorithm}
\Cref{alg:reducibility} runs in polynomial time, and its output is correct.
\end{theorem}
\begin{proof}
We start by showing that the algorithm runs in polynomial time. The outer \textbf{for} loop runs $O(m^2)$ times, and the inner \textbf{while} loop runs at most $m$ times. Each basic operation can be implemented in polynomial time, and so the entire algorithm runs in polynomial time.

Suppose first that the algorithm outputs ``reducible''. By construction, all subcubes in $F \setminus G$ are disjoint from $\bigvee G$. Since $F$ is a subcube partition, this means that $\bigcup G = \bigvee G$, which is a subcube. By construction, $1 < |G| < |F|$, and so $F$ is indeed reducible.

To complete the proof, we show that if $F$ is reducible, then the algorithm outputs ``reducible''.
If $F$ is reducible then there is a subset $H \subset F$, with $1 < |H| < |F|$, such that $\bigcup H$ is a subcube. Let $s_i, s_j \in H$, and consider the $(i,j)$ iteration of the outer \textbf{for} loop.

We prove inductively that at each iteration of the inner \textbf{while} loop, $G$ is contained in $H$. This holds by construction at the very first step. Now suppose that $G \subseteq H$ and that $s_k \notin G$ intersects $\bigvee G$. Since $G \subseteq H$, also $\bigvee G \subseteq \bigvee H$, and so $s_k$ intersects $\bigvee H$. Since $\bigvee H = \bigcup H$ and the subcubes in $F$ are disjoint, necessarily $s_k \in H$. Hence $G \cup \{s_k\} \subseteq H$.

When the \textbf{while} loop ends, all $s_k \notin G$ are disjoint from $\bigvee G$. Since the subcubes in $F$ are disjoint, this means that $\bigvee G = \bigcup G$. Since $G \subseteq H$, necessarily $G \neq F$, and so the algorithm correctly declares that $F$ is reducible.
\end{proof}

\subsection{Parity argument}
\label{sec:kis-construction}

In this section we present two constructions of infinite families of tight irreducible subcube partitions.

\begin{theorem} \label{thm:cubic}
Let $n = 2m + 1 \ge 3$. The following subcubes comprise a tight irreducible subcube partition of size $\Theta(n^3)$:
\begin{itemize}
\item The point $0^n$.
\item All cyclic rotations of $0^m1*^m$.
\item For every $0 \le i,j,k \le m-1$ satisfying $i+j,j+k \le m-1$, the subcube \[ 0^i 1 *^j 0^k 1 *^{m-1-j-k} 0^j 1 *^{m-1-i-j}. \]
\end{itemize}
\end{theorem}

We found this subcube partition by starting with the subcube partition consisting of all rotations of $0^m 1 *^m$ together with all points not covered by them. This subcube partition is reducible, and we can use \Cref{alg:reducibility} to merge together points into subcubes. One can show inductively that the rotations of $0^m1*^m$ never get merged, and so the resulting subcube partition is not trivial. It is precisely the one described in \Cref{thm:cubic}.

Here is the resulting partition for $n = 5$:
\[
\begin{array}{cccccc}
00000 &
001** & *001* & **001 & 1**00 & 01**0 \\ &
11*1* & 1011* & 011*1 & 01011 & 1*101
\end{array}
\]

\begin{theorem} \label{thm:LagariasShor}
Let $n \ge 1$ be odd. The following subcubes comprise a tight irreducible subcube partition of size $F_{n+1} + F_{n-1} + 1$, where $F_n$ is the $n$'th Fibonacci number:
\begin{itemize}
\item The point $1^n$.
\item All subcubes obtained by concatenating blocks of the form $*1$ or $0$ in an arbitrary way, and rotating the result arbitrarily.
\end{itemize}
\end{theorem}

This subcube partition appears in \cite[Section 2]{Kisielewicz14}, where it is attributed to Lagarias and Shor~\cite{LagariasShor94}. Here is the partition for $n = 5$:
\[
\begin{array}{cccccc}
00000 & *1000 & 0*100 & 00*10 & 000*1 & 1000* \\
11111 & *1*10 & *10*1 & 0*1*1 & 1*10* & 10*1*
\end{array}
\]

In both cases, we will prove the irreducibility using the following \namecref{lem:parity}, suggested to us by Kisielewicz.

\begin{definition}[Star pattern] \label{def:star-pattern}
The \emph{star pattern} of a subcube $s \in \{0,1,*\}^n$ is $P(s) := \{ i \in [n] : s_i = * \}$.
\end{definition}

\begin{definition}[Parity of a subcube] \label{def:subcube-parity}
The \emph{parity} of a subcube $s \in \{0,1,*\}^n$ is the parity of the number of $1$s in $s$.
\end{definition}

\begin{lemma}[{\cite{Kisielewicz23}}] \label{lem:parity}
Let $F$ be a subcube partition. Let $G$ be a subset of $F$ such that $|G| > 1$ and the union of $G$ is a subcube. Let $S$ be an inclusion-minimal star pattern occurring in $G$ (this means that no star pattern strictly contained in $S$ appears in $G$).

Among subcubes in $G$ whose star pattern is $S$, half have even parity and half have odd parity.
\end{lemma}
\begin{proof}
Let $C \subseteq \bigcup G$ consist of all points $p$ such that $p_i = 0$ for all $i \in S$.
If $|C| = 1$ then $S$ is the star pattern of $\bigcup G$, which contradicts $|G| > 1$. Hence $|C| > 1$. Since $C$ is a subcube, it contains an equal number of points of even parity and of odd parity.

If $s \in G$ has star pattern other than $S$ then by inclusion-minimality, $s_i = *$ for some $i \notin S$. Therefore $s \cap C$ contains an equal number of points of even parity and of odd parity.
In contrast, if $s \in G$ has star pattern $S$ then $s \cap C$ contains a single point whose parity is the same as the parity of $s$. Since $C$ has an equal number of points of either parity, the \namecref{lem:parity} immediately follows.
\end{proof}

\begin{corollary} \label{cor:parity}
Let $F$ be a subcube partition in which there is a star pattern $S$ occurring twice, and every other star pattern occurs at most once. If $G$ is a subset of $F$ whose union is a subcube and $|G| > 1$ then $G$ contains both subcubes of $F$ whose star pattern is $S$.
\end{corollary}
\begin{proof}
Let $U$ be an inclusion-minimal star pattern in $G$. \Cref{lem:parity} implies that $G$ must contain an even number of subcubes whose star pattern is $U$. Necessarily $U = S$, and so $G$ contains both subcubes whose star pattern in $S$.
\end{proof}

\Cref{cor:parity} almost immediately implies the irreducibility of the subcube partition in \Cref{thm:LagariasShor}: any non-singleton subset of the subcube partition whose union is a subcube must contain both points $0^n,1^n$, and so its union must be $\{0,1\}^n$. The argument for \Cref{thm:cubic} is only slightly longer.

We prove \Cref{thm:cubic} in \Cref{sec:cubic}, and \Cref{thm:LagariasShor} in \Cref{sec:LagariasShor}.

\subsubsection{Cubic construction} \label{sec:cubic}

In this section we prove \Cref{thm:cubic}.

We need to prove three things about the set of subcubes $F$ given in the statement of the \namecref{thm:cubic}: that it is a subcube partition; that it is tight; and that it is irreducible.

\paragraph{Subcube partition}
The point $0^n$ covers itself. All other subcubes of $F$ contain at least one $1$.

Subcubes of the second type cover \emph{royal points}. These are points which contain a \emph{royal $1$}, which is a $1$ preceded cyclically by $m$ many $0$s. Since $n < 2(m+1)$, there can be at most one royal $1$, and so royal points are covered by precisely one subcube of the second type. We will soon see that they are not covered by any subcube of the third type.

We can guarantee that a subcube does not contain any royal point by adding ``blocking $1$s'': if each cyclic interval of length $m$ contains a $1$, then the subcube cannot contain any royal point. Each subcube of the third type is contained in the subcube $*^i 1 *^{j+k} 1 *^{m-1-k} 1 *^{m-1-i-j}$, in which the $1$s are separated by $j+k,m-1-k,m-1-j \le m-1$ many stars. Consequently, each royal point is covered by precisely one subcube of $F$.

It remains to handle points $x \neq 0^n$ which are not royal. Let $I+1$ be the index of the first $1$ in $x$. Since $x$ is not royal, $I \leq m-1$.

Let $I + 1 + m + J + 1$ be the first $1$ in $x$ beyond position $I + 1 + m$ (so $J \ge 0$). Such a $1$ exists since otherwise $x$ starts with $0^I 1$ and ends with $0^{m-I}$, and is consequently royal. For the same reason, $J \le m-1$. Since $I + 1 + m + J + 1 \leq n$, we see that $I + J \leq m - 1$.

Let $I + 1 + J + K + 1$ be the first $1$ in $x$ beyond position $I + 1 + J$ (so $K \ge 0$). Such a $1$ exists as seen before. Since $x$ is not royal, $J + K \le m - 1$. Collecting all the information, we see that $x$ belongs to the subcube of the third type
{
\begin{gather*}
 0^I 1 *^J \tikzmark{a} 0^K 1 *^{m - 1 - J - K} \tikzmark{b} 0^J 1 *^{m - 1 - I - J}. \\
\end{gather*}
\begin{tikzpicture}[overlay, remember picture]
 \node (aa) [below of = a, node distance = 2 em, anchor = south] {\footnotesize $I+1+J$};
 \node (bb) [below of = b, node distance = 2 em, anchor = south] {\footnotesize $I+1+m$};
 \draw[->] (aa.north) to (a.south);
 \draw[->] (bb.north) to (b.south);
\end{tikzpicture}
}
If $x \in 0^i 1 *^j 0^k 1 *^{m-1-j-k} 0^j 1 *^{m-1-i-j}$ and we follow the steps above then we find that $i = I$, $j = J$, and $k = K$. Therefore $x$ belongs to a unique subcube of $F$.

\paragraph{Tightness} This is clear, since the subcube $0^n$ mentions all coordinates.

\paragraph{Irreducibility} We will use \Cref{cor:parity} in order to prove irreducibility, so we first need to understand the star patterns of the various subcubes in $F$.

A subcube of the second type has precisely $m$ stars, and each subcube of the second type has a different star pattern. Furthermore, each star pattern either consists of a single interval, or of one interval starting at $1$ and another interval ending at $n$.

A subcube of the third type has no interval of $m$ stars, and cannot start with a star, hence its star patterns differ from those of subcubes of the second type.
Given the star pattern of a subcube of the third type, we can determine $i,j,k$. First, we determine $i + j$ by counting the number of trailing stars, which is $m-1-i-j$. This allows us to determine $j$ (and so $i$), by counting the number of stars in the first $i + 1 + j$ symbols. We can now determine $k$ by counting the number of stars in the first $i + 1 + m$ symbols, which is $m-1-k$.

Summarizing, if we consider only subcubes of the second and third types, then all star patterns are unique. Considering the entire subcube partition, there is one star pattern occurring twice, corresponding to the points $0^n$ and $0^{m-1}10^{m-1}11$, and all other star patterns occur once.

If $F$ were reducible then there would be a subset $G \subseteq F$ such that $1 < |G| < |F|$ and the union of $G$ is a subcube. According to \Cref{cor:parity}, $G$ must contain both points $0^n,0^{m-1}10^{m-1}11$, and so $\bigcup G$ must contain their join $0^{m-1}*0^{m-1}**$.
This implies that $G$ must cover the point $0^{m-1}10^{m-1}00$, and so must contain the subcube $0^{m-1}1*^m0$. Similarly, it must cover the point $0^{m-1}00^{m-1}10$, and so contain the subcube $*^{m-1}0^m1*$. Since the join of the latter two subcubes is $*^n$, we conclude that $G = F$, contrary to assumption. Therefore $F$ is irreducible.

\subsubsection{Lagarias--Shor construction} \label{sec:LagariasShor}

In this section we prove \Cref{thm:LagariasShor}.

We need to prove three things about the set of subcubes $F$ given in the statement of the \namecref{thm:LagariasShor}: that it is a subcube partition; that it is tight; and that it is irreducible. (We leave it to the reader to prove the formula for the size of $F$.)

\paragraph{Subcube partition} The point $1^n$ covers itself. Since $n$ is odd, every other subcube contains $0$, and so doesn't cover $1^n$.

Consider now a point $x \in \{0,1\}^n$ other than $1^n$. We will convert it, in stages, to a subcube $y \in F$ which contains it.

We initialize $y$ with $x$.
If there is an index $i$ such that $y_i = 1$ and $y_{i+1} \neq 1$ (treating indices cyclically) then any subcube in $F$ which covers $y$ must have $y_i = 1$ and so $y_{i-1} = *$. Accordingly, as long as there is an index $i$ such that $y_{i-1} \neq *$, $y_i = 1$, $y_{i+1} \neq 1$,
we set $y_{i-1} = *$.

When the process stops, every $1$ is either preceded by $*$ or followed by $1$. Since $x \neq 1^n$, every run of $1$ in $y$ must terminate (followed by $*$ or $0$). The final $1$ in each such run is not followed by $1$, and so must be preceded by $*$ (implying that the run has length $1$). It follows that $y$ is, up to rotation, a concatenation of copies of $*1$ and $0$, and so $y \in F$. Furthermore, the construction of $y$ ensures that this is the only subcube in $F$ covering $x$.

\paragraph{Tightness} This is clear, since the subcube $1^n$ mentions all coordinates.

\paragraph{Irreducibility} In view of using \Cref{cor:parity}, we first explore the star patterns of the various subcubes in $F$. The main observation is that we can recover a subcube of the second type from its star pattern. Indeed, every star must be followed by $1$, and every position not preceded by a star must be $0$.
This implies that apart from the two points $0^n,1^n$, all other star patterns are unique.

If $F$ were reducible then there would exist a subset $G \subset F$ with $1 < |G| < |F|$ whose union is a subcube. \Cref{cor:parity} shows that $G$ must contain both points $0^n,1^n$, and so their join $*^n$, contradicting the assumption $G \neq F$. Hence $F$ is irreducible.

\subsection{Minimal size}
\label{sec:minimal-subcubes}

What is the minimal size of a tight irreducible subcube partition of length $n$? (The question doesn't make sense without assuming tightness, since $*^n$ is always irreducible.)

When $n = 1$, there is a single tight irreducible subcube partition: $0, 1$. When $n = 2$, there are no tight irreducible subcube partitions. When $n = 3$, there is a unique tight irreducible subcube partition, up to flipping and rearranging coordinates:
\[
 000, *01, 1*0, 01*, 11.
\]
For $n = 4, 5, 6, 7$, Peitl and Szeider~\cite{PeitlSzeider22} used a computer search to show that the minimal number of subcubes is $7, 9, 11, 13$, respectively. This is consistent with the following conjecture.

\begin{conjecture} \label{conj:minimal-subcubes}
If $n \ge 3$ then the minimal size of a tight irreducible subcube partition of length $n$ is $2n - 1$.
\end{conjecture}

\Cref{sec:size-lb} explains the best lower bound on the size, due to Kullmann and Zhao~\cite{KullmannZhao13}.
\Cref{sec:merging,sec:twisting} present two constructions of an infinite family of tight irreducible subcube partitions of length $n$ and size $2n-1$. In \Cref{sec:minimal-weight} we present several more such constructions which will be useful in \Cref{sec:larger-alphabets}.

\subsubsection{Lower bound} \label{sec:size-lb}

Before presenting the constructions of tight irreducible subcube partitions of size $2n-1$, here is the best lower bound on the size, due to Kullmann and Zhao~\cite{KullmannZhao16}. We give an alternative proof using known results from the literature.

\begin{theorem} \label{thm:size-lower-bound}
If $n \ge 4$ then every tight irreducible subcube partition of length $n$ has size at least $n + 3$.
\end{theorem}

Before proving the theorem, we need a simple \namecref{lem:irreducible-regular}.

\begin{definition}[Regularity] \label{def:regular}
A subcube partition of length $n$ is \emph{regular} if for every $i \in [n]$ and every $b \in \{0,1\}$ there are at least two subcubes $s \in F$ such that $s_i = b$.
\end{definition}

This definition is due to Kullmann and Zhao~\cite{KullmannZhao13}, who used the term \emph{nonsingular}. The term \emph{regular} appears in Peitl and Szeider~\cite{PeitlSzeider22}.

\begin{lemma}[{\cite[Lemma 39]{KullmannZhao16}}] \label{lem:irreducible-regular}
If $F$ is a tight irreducible subcube partition of length $n \ge 2$ then $F$ is regular.
\end{lemma}
\begin{proof}
We prove the definition of regularity for $i = 1$.

For $\sigma \in \{0,1,*\}$, let $F_\sigma = \{ x : \sigma x \in F \}$. Both $F_0 \cup F_*$ and $F_1 \cup F_*$ are subcube partitions of length $n - 1$, and so $\bigcup F_0 = \bigcup F_1$. Since $F$ is tight, $F_0,F_1$ are non-empty.

If $F_0 = \{x\}$ and $|F_1| > 1$ then the union of the subcubes corresponding to $F_1$ is the subcube $1x$, contradicting irreducibility.

If $F_0 = \{x\}$ and $|F_1| = 1$ then $F_0 = F_1 = \{x\}$ and so the union of the corresponding subcubes is $*x$. Since $F$ is irreducible, necessarily $x = *^{n-1}$, and so $F = \{ 0*^{n-1}, 1*^{n-1} \}$. Since $F$ is tight, necessarily $n = 1$, contradicting the assumption $n \ge 2$. 

It follows that $|F_0| \ge 2$. Similarly $|F_1| \ge 2$.
\end{proof}

We can now prove the size lower bound.

\begin{proof}[Proof of \Cref{thm:size-lower-bound}]
Let $F = \{ s_1,\dots,s_m \}$ be a tight subcube partition of length $n$. We can identify $F$ with a formula $\Phi$ in \emph{conjunctive normal form} (CNF) over variables $x_1,\dots,x_n$ whose clauses are ``$x \notin s_i$'' for all $i \in [m]$. For example, the subcube partition $0*,10,11$ corresponds to the CNF $x_1 \land (\bar{x}_2 \lor x_3) \land (\bar{x}_2 \lor \bar{x}_3)$.

Since every $x$ belongs to some $s_i$, the formula $\Phi$ is unsatisfiable. It is moreover minimally unsatisfiable, meaning that if we remove any clause, then it becomes satisfiable. Indeed, if we remove the clause ``$x \notin s_i$'', then any point in $s_i$ would satisfy the formula. Since $F$ is tight, $\Phi$ mentions all $n$ variables.

A well-known result attributed to Tarsi~\cite{AharoniLinial86} states that a minimally unsatisfiable CNF mentioning $n$ variables must contain at least $n + 1$ clauses, hence $m \ge n+1$.

Suppose that $m = n+1$. Davydov, Davydova, and Kleine~Büning~\cite[Theorem 12]{DDKB98} showed that if a minimally unsatisfiable CNF mentioning $n$ variables contains exactly $n + 1$ clauses, then some variable appears once positively and once negatively. In particular, $F$ is not regular, contradicting \Cref{lem:irreducible-regular}. Hence $m \ge n+2$.

Suppose that $m = n+2$. Kleine~Büning~\cite[Theorem 6]{KB00} showed that there is a unique regular minimally unsatisfiable CNF mentioning $n$ variables which contains exactly $n + 2$ clauses, up to renaming and reordering variables. The collection of subcubes corresponding to this CNF consists of $0^n,1^n$ together with all cyclic rotations of $10*^{n-2}$. When $n \ge 4$, these subcubes are not disjoint: for example, $10*^{n-2}$ and $*^210*^{n-4}$ both contain the subcube $1010*^{n-4}$. Hence $m \ge n+3$.
\end{proof}

In the following two subsections, we present two constructions of the same sequence of tight irreducible subcube partitions of length $n \ge 3$ and size $2n-1$.

\subsubsection{Merging} \label{sec:merging}

Our first construction is based on the following \namecref{lem:merge}, which is used to merge together two subcube partitions.

\begin{definition}[Reducibility for partial subcube partitions] \label{def:reducibility-partial}
A subset $F'$ of a subcube partition $F$ of length $n$ is \emph{reducible} if there exists a subset $G \subseteq F'$, with $|G| > 1$, such that the union of the subcubes in $G$ is a subcube different from $*^n$.
\end{definition}

\begin{lemma} \label{lem:merge}
Let $F_0,F_1$ be two subcube partitions of length $n$. Let
\[
 G = \{ 0x : x \in F_0 \setminus F_1 \} \cup \{ 1x : x \in F_1 \setminus F_0 \} \cup \{ *x : x \in F_0 \cap F_1 \}.
\]
Then
\begin{enumerate}[(a)]
\item $G$ is a subcube partition of length $n+1$.
\item If $F_0 \neq F_1$ and at least one of them is tight, then $G$ is tight.
\item If $F_0 \cap F_1 \neq \emptyset$ and both $F_0 \setminus F_1$ and $F_1$ are irreducible (or both $F_1 \setminus F_0$ and $F_0$ are irreducible) then $G$ is irreducible.
\end{enumerate}
\end{lemma}

\begin{proof}
The first two items follow easily from the construction (the condition $F_0 \neq F_1$ in the second item guarantees that the first coordinate is mentioned).

Now suppose that $F_0 \cap F_1 \neq \emptyset$ and both $F_0 \setminus F_1$ and $F_1$ are irreducible. We need to show that $G$ is irreducible. If not, then there is a subset $H \subset G$, with $1 < |H| < |G|$, whose union is a subcube $x \neq *^n$.

If $x = 0y$ then $y$ is a union of $|H|$ subcubes in $F_0 \setminus F_1$. Since $F_0 \setminus F_1$ is irreducible and $|H| > 1$, necessarily $y = *^n$. However, this contradicts the assumption $F_0 \cap F_1 \neq \emptyset$.

We get a similar contradiction if $x = 1y$, using the irreducibility of $F_1$.

Finally, if $x = *y$ then $y$ is a union of $|H|$ subcubes of $F_0$ as well as a union of $|H|$ subcubes of $F_1$. Since $F_1$ is irreducible and $y \neq *^n$, necessarily $y \in F_1$. If $y \in F_0$ then $x \in G$, contradicting the assumption $|H| > 1$. If $y \notin F_0$ then $1y \in G$ and so $y$ is a union of subcubes in $F_0 \setminus F_1$. Since $F_0 \setminus F_1$ is irreducible and $y \neq *^n$, necessarily $y \in F_0 \setminus F_1$, contradicting both $y \notin F_0$ and $y \in F_1$.
\end{proof}

We now construct the promised sequence of tight irreducible subcube partitions.

\begin{theorem} \label{thm:size-2n-1-merging}
For each $n \ge 3$ there is a tight irreducible subcube partition $S_n$ of length $n$ and size $2n-1$.
\end{theorem}
\begin{proof}
We construct the subcube partitions inductively. The starting point is
\[
 S_3 = \{ 000, *01, 1*0, 01*, 111 \},
\]
whose irreducibility was proved by Kullmann and Zhao~\cite[Lemma 41]{KullmannZhao16}, and can also be checked using \Cref{alg:reducibility}.
The construction will maintain the invariants that $01*^{n-2} \in S_n$ and $1*^{n-1},00*^{n-2} \notin S_n$, and moreover $|S_n| = 2n-1$

Given $S_n$, we construct $S_{n+1}$ by applying \Cref{lem:merge} to $F_0 = \{1*^{n-1}, 00*^{n-2}, 01*^{n-2}\}$ and $F_1 = S_n$.

% (we are using padding)

Since $F_0$ is reducible and $F_1$ is irreducible, clearly $F_0 \neq F_1$, and so $S_{n+1}$ is tight by \Cref{lem:merge}.

The invariant implies that $F_0 \setminus F_1 = \{ 1*^{n-1}, 00*^{n-2} \}$ is irreducible. It follows that $S_{n+1}$ is irreducible by \Cref{lem:merge}.

Since $1*^{n-1} \notin F_1$, it follows that $01*^{n-1} \in S_{n+1}$. Since $0*^{n-1} \notin F_0$, it follows that $00*^{n-1} \notin S_{n+1}$. Since $|F_1| > 1$, in particular $*^n \notin F_1$, and so $1*^n \notin S_{n+1}$.

Finally, the invariants imply that $F_0 \cap F_1 = \{ 01*^{n-2} \}$, and so
\[
 |S_{n+1}| = |F_0 \setminus F_1| + |F_1 \setminus F_0| + |F_0 \cap F_1| = 2 + (|F_1| - 1) + 1 = 2n+1,
\]
using $|F_1| = 2n-1$.
\end{proof}

Here are the resulting subcube partitions for $n = 3,4,5$:
{
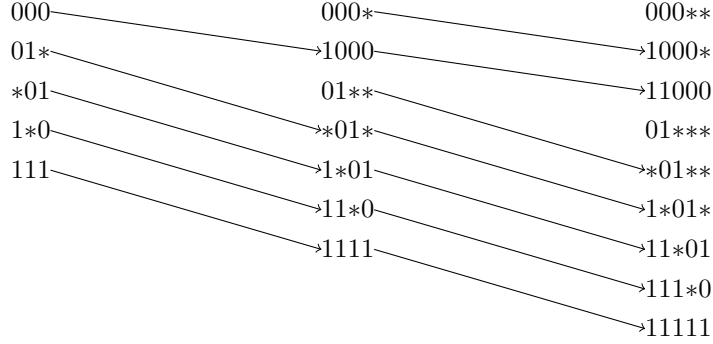
\begin{align*}
&000\tikzmark{a} & &000*\tikzmark{f} & &000** \\
&01*\tikzmark{b} & &\tikzmark{aa}1000\tikzmark{g} & & \tikzmark{ff}1000* \\
&*01\tikzmark{c} & &01**\tikzmark{h} & &\tikzmark{gg}11000 \\
&1*0\tikzmark{d} & &\tikzmark{bb}*01*\tikzmark{i} & &01*** \\
&111\tikzmark{e} & &\tikzmark{cc}1*01\tikzmark{j} & &\tikzmark{hh}*01** \\
&    & &\tikzmark{dd}11*0\tikzmark{k} & &\tikzmark{ii}1*01* \\
&    & &\tikzmark{ee}1111\tikzmark{l} & &\tikzmark{jj}11*01 \\
&    & &     & &\tikzmark{kk}111*0 \\
&    & &     & &\tikzmark{ll}11111
\end{align*}
\begin{tikzpicture}[overlay, remember picture]
\draw [->] ([yshift=3pt]a.east) to ([yshift=3pt]aa.west);
\draw [->] ([yshift=3pt]b.east) to ([yshift=3pt]bb.west);
\draw [->] ([yshift=3pt]c.east) to ([yshift=3pt]cc.west);
\draw [->] ([yshift=3pt]d.east) to ([yshift=3pt]dd.west);
\draw [->] ([yshift=3pt]e.east) to ([yshift=3pt]ee.west);
\draw [->] ([yshift=3pt]f.east) to ([yshift=3pt]ff.west);
\draw [->] ([yshift=3pt]g.east) to ([yshift=3pt]gg.west);
\draw [->] ([yshift=3pt]h.east) to ([yshift=3pt]hh.west);
\draw [->] ([yshift=3pt]i.east) to ([yshift=3pt]ii.west);
\draw [->] ([yshift=3pt]j.east) to ([yshift=3pt]jj.west);
\draw [->] ([yshift=3pt]k.east) to ([yshift=3pt]kk.west);
\draw [->] ([yshift=3pt]l.east) to ([yshift=3pt]ll.west);
\end{tikzpicture}
}

\subsubsection{Twisting} \label{sec:twisting}

Our second construction starts with the observation
\[
 x00 \cup x1* = x*0 \cup x11.
\]
Up to permutation and flipping of coordinates, this is the only way in which a set of points can be written as a union of two subcubes in two different ways, as we show below in \Cref{lem:nfs-pair}. Following Kullmann and Zhao~\cite[Definitions 45--46]{KullmannZhao16}, we call such a pair of subcubes an \emph{nfs-pair}. The \emph{nfs-flip} of the pair on the left is the pair on the right.

\begin{definition}[Nfs-pair, nfs-flip] \label{def:nfs-flip}
Two subcubes $s,t$ constitute an \emph{nfs-pair} if they differ on exactly two positions $i,j$, where $(s_i,t_i) \in \{(0,1),(1,0)\}$ and $t_j = *$.

The \emph{nfs-flip} of $s,t$ is the pair of subcubes $s',t'$ obtained by copying the coordinates except for $i,j$, and setting $s'_i = *$, $s'_j = s_j$, $t'_i = t_i$, $t'_j = \bar{s}'_j$.
\end{definition}

\begin{lemma} \label{lem:nfs-flip}
If $s,t$ is an nfs-pair with nfs-flip $s',t'$ then $s \cup t = s' \cup t'$.
\end{lemma}

Nfs-pairs are the only pairs of subcubes satisfying \Cref{lem:nfs-flip} non-trivially.

\begin{lemma} \label{lem:nfs-pair}
Let $s,t$ and $s',t'$ be two pairs of disjoint subcubes such that $s \cup t = s' \cup t'$, the common value is not a subcube, and $\{s,t\} \neq \{s',t'\}$. Then either $s,t$ or $t,s$ is an nfs-pair, and $s',t'$ or $t',s'$ is its nfs-flip.
\end{lemma}
\begin{proof}
Since $s,t$ are disjoint, they must conflict. Without loss of generality, $s = 0p$ and $t = 1q$. If both $s'$ and $t'$ start with non-stars then clearly $\{s,t\} = \{s',t'\}$, and if both start with a star then $p = q$ and so $s \cup t = *p$ is a subcube. Therefore without loss of generality, $s' = *p'$ and $t' = 1q'$.

Since $0p \cup 1q = *p' \cup 1q'$, considering the points starting with $0$, we see that $p' = p$. Considering the points starting with $1$, we see that $q = p \cup q'$. Since $p,q',q$ are all subcubes, it must be that $p,q'$ are subcubes differing in a single non-star position, and $q$ is obtained from them by changing this position to a star. Thus $s,t$ is an nfs-pair, and $s',t'$ is its nfs-flip.
\end{proof}

The construction is based on the following simple corollary of \Cref{lem:merge}.

\begin{lemma} \label{lem:twist}
Let $F$ be a tight irreducible subcube partition containing an nfs-pair $s,t$, and let $s',t'$ be its nfs-flip. The following subcube partition is tight and irreducible, for any $b \in \{0,1\}$:
\[
 G = \{ *x : x \in F, x \neq s,t \} \cup \{ bs, bt, \bar{b}s', \bar{b}t' \}.
\]

Furthermore, $|G| = |F| + 2$, and $G$ contains the nfs-pairs $bs,bt$ and $\bar{b}s',\bar{b}t'$.
\end{lemma}
\begin{proof}
Let $F'$ be the formula obtained from $F$ by replacing $s,t$ with $s',t'$. We apply \Cref{lem:merge} on $F_b = F$ and $F_{\bar{b}} = F'$, obtaining the stated subcube partition $G$.

Since $F \neq F'$ and $F$ is tight, $G$ is tight.

Clearly $F$ cannot consist only of $s,t$, and so $F \cap F' \neq \emptyset$. Since $F$ and $F' \setminus F \subset F'$ are both irreducible, it follows that $G$ is irreducible.
\end{proof}

In order to obtain the sequence $S_n$ constructed in \Cref{thm:size-2n-1-merging} using \Cref{lem:twist}, start with $S_3$. Given $S_n$, apply the \namecref{lem:twist} with $s=1^n$, $t=1^{n-2}*0$, and $b = 1$, and rotate the resulting subcube partition once to the left. The result is $S_{n+1}$. Here is an example:
{
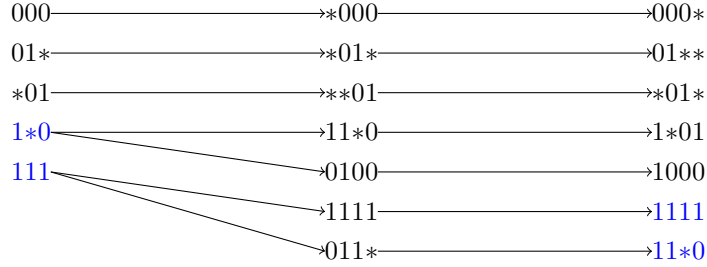
\begin{align*}
&000 \tikzmark{a} & &\tikzmark{aa} *000 \tikzmark{f} & & \tikzmark{ff} 000* \\
&01* \tikzmark{b} & &\tikzmark{bb} *01* \tikzmark{g} & &\tikzmark{gg} 01** \\
&*01 \tikzmark{c} & &\tikzmark{cc} **01 \tikzmark{h} & &\tikzmark{hh} *01* \\
&\lmarked{1*0} \tikzmark{d} & &\tikzmark{dd} 11*0 \tikzmark{i} & &\tikzmark{ii} 1*01 \\
&\lmarked{111} \tikzmark{e} & &\tikzmark{ddd} 0100 \tikzmark{j} & &\tikzmark{jj} 1000 \\
& & &\tikzmark{ee} 1111\tikzmark{k} & &\tikzmark{kk} \lmarked{1111} \\
& & &\tikzmark{eee} 011*\tikzmark{l} & &\tikzmark{ll} \lmarked{11*0}
\end{align*}
\begin{tikzpicture}[overlay, remember picture]
\draw [->] ([yshift=3pt]a.east) to ([yshift=3pt]aa.west);
\draw [->] ([yshift=3pt]b.east) to ([yshift=3pt]bb.west);
\draw [->] ([yshift=3pt]c.east) to ([yshift=3pt]cc.west);
\draw [->] ([yshift=3pt]d.east) to ([yshift=3pt]dd.west);
\draw [->] ([yshift=3pt]d.east) to ([yshift=3pt]ddd.west);
\draw [->] ([yshift=3pt]e.east) to ([yshift=3pt]ee.west);
\draw [->] ([yshift=3pt]e.east) to ([yshift=3pt]eee.west);
\draw [->] ([yshift=3pt]f.east) to ([yshift=3pt]ff.west);
\draw [->] ([yshift=3pt]g.east) to ([yshift=3pt]gg.west);
\draw [->] ([yshift=3pt]h.east) to ([yshift=3pt]hh.west);
\draw [->] ([yshift=3pt]i.east) to ([yshift=3pt]ii.west);
\draw [->] ([yshift=3pt]j.east) to ([yshift=3pt]jj.west);
\draw [->] ([yshift=3pt]k.east) to ([yshift=3pt]kk.west);
\draw [->] ([yshift=3pt]l.east) to ([yshift=3pt]ll.west);
\end{tikzpicture}
}

\subsection{Minimal weight}
\label{sec:minimal-weight}

In \Cref{sec:larger-alphabets}, we will consider irreducible subcube partitions over larger alphabets. As we show in \Cref{sec:expansion}, one of the ways to construct an irreducible subcube partition over an alphabet $\{0,\dots,q-1\}$ is to start with an irreducible subcube partition over $\{0,1\}$, and replace each $1$ in each subcube with each of $\{1,\ldots,q-1\}$. The resulting number of subcubes is
\[
 \sum_{s \in F} (q-1)^{\nones{s}},
\]
where $F$ is the subcube partition we start with, and $\nones{s}$ is the number of $1$s in $s$. This suggests looking for a tight irreducible subcube partition which minimizes the above objective function.

The concept of \emph{majorization} allows us to optimize this objective function for all $q$'s at once.

\begin{definition}[Weight vector] \label{def:weight-vector}
Let $F$ be a subcube partition of length $n$. Its \emph{weight vector} is the vector $w(F) = w_0,\dots,w_n$, where $w_h$ is the number of subcubes of $F$ of \emph{weight} $h$, that is, with $h$ many $1$s.

The notation $w_{\ge h}$ stands for $w_h + \cdots + w_n$, which is the number of subcubes with at least $h$ many $1$s.
\end{definition}

\begin{definition}[Majorization] \label{def:majorization}
Let $a,b$ be two weight vectors of length $n+1$. We say that $a$ \emph{majorizes} $b$ if for every $h \leq n$, we have $a_{\geq h} \geq b_{\geq h}$.
\end{definition}

\begin{lemma} \label{lem:majorization}
Let $F,G$ be subcube partitions of length $n$. If $w(F)$ majorizes $w(G)$ then for all monotone non-decreasing functions $\phi\colon \{0,\dots,n\} \to \mathbb{R}$,
\[
 \sum_{s \in F} \phi(\nones{s}) \geq \sum_{s \in G} \phi(\nones{s}).
\]
In particular, this holds for $\phi(h) = (q-1)^h$ as long as $q \ge 2$.
\end{lemma}
\begin{proof}
We will show that
\[
 \sum_{h=0}^n w_h(F) \phi(h) \ge \sum_{h=0}^n w_h(G) \phi(h).
\]
Indeed,
\begin{align*}
 \sum_{h=0}^n w_h(F) \phi(h) &=
 w_{\ge 0}(F) \phi(0) + \sum_{h=1}^n w_{\ge h}(F) (\phi(h) - \phi(h-1)) \\ &\ge
 w_{\ge 0}(G) \phi(0) + \sum_{h=1}^n w_{\ge h}(G) (\phi(h) - \phi(h-1)) =
 \sum_{h=0}^n w_h(G) \phi(h). \qedhere
\end{align*}
\end{proof}

% \begin{lemma} \label{lem:majorization}
% Let $F,G$ be subcube partitions of length $n$. If $w(F)$ majorizes $w(G)$ then for all $q \geq 2$,
% \[
%  \sum_{s \in F} (q-1)^{\nones{s}} \geq \sum_{s \in G} (q-1)^{\nones{s}}.
% \]
% \end{lemma}
% \begin{proof}
% Let $r = q-1 \geq 1$. We will show that
% \[
%  \sum_{h=0}^n w_h(F) r^h \geq \sum_{h=0}^n w_h(G) r^h.
% \]
% Indeed,
% \[
%  \sum_{h=0}^n w_h(F) r^h =
%  w_{\ge 0}(F) + 
%  \sum_{h=1}^n w_{\geq h}(F) (r^h - r^{h-1}) \geq
%  w_{\geq 0}(G) +
%  \sum_{h=1}^n w_{\geq h}(G) (r^h - r^{h-1}) =
%  \sum_{h=0}^n w_h(G) r^h. \qedhere
% \]
% \end{proof}

\Cref{lem:majorization} allows us to reformulate our goal: find the minimal weight vectors (in the sense of majorization) of the tight irreducible subcube partitions of length $n$. (There could be more than one minimal weight vector, since majorization is not a linear order.)

\begin{conjecture} \label{conj:minimal-weight}
For every $n \geq 3$, the minimal weight vectors of tight irreducible subcube partitions of length $n$ are $1,n-1,n-1,0,\dots,0$ and $1,n,n-3,1,0,\dots,0$.
\end{conjecture}

In \Cref{sec:weight-lb}, we show that \Cref{conj:minimal-subcubes} implies the lower bound part of \Cref{conj:minimal-weight}. In \Cref{sec:weight-construction} we give matching constructions.

\smallskip

Unconditionally, we can show that every tight irreducible subcube partition of length $n \ge 3$ must contain a subcube of weight~$2$.

\begin{lemma} \label{lem:minimal-maximum-weight}
If $F$ is a tight irreducible subcube partition of length $n \ge 3$ then $F$ contains a subcube of weight at least $2$.
\end{lemma}
\begin{proof}
Suppose that every subcube in $F$ has weight at most $1$. Let $s \in F$ be the subcube containing $1^n$. If $s$ has weight $0$ then $s = *^n$, contradicting the tightness of $F$. If $s$ has weight $1$ then, without loss of generality, $s = 1*^{n-1}$. The union of all other subcubes of $F$ must be $0*^{n-1}$, and so by irreducibility, $F = \{ 0*^{n-1}, 1*^{n-1} \}$, contradicting tightness.
\end{proof}

\subsubsection{Lower bound} \label{sec:weight-lb}

In this section we prove the lower bound part of \Cref{conj:minimal-weight}, assuming \Cref{conj:minimal-subcubes}. As we explain in the proof, this amounts to ruling out the weight vector $1,n,n-2,0,\dots,0$.

\begin{theorem} \label{thm:minimal-weight-lb}
Assume \Cref{conj:minimal-subcubes}. For every $n \geq 3$, the weight vector of any tight irreducible subcube partition of length $n$ majorizes either $1,n-1,n-1,0,\dots,0$ or $1,n,n-3,1,0,\dots,0$.
\end{theorem}
\begin{proof}
Let $F$ be a tight irreducible subcube partition of length $n$, and let $w$ be its weight vector. The theorem states that (i) $w_{\ge 0} \geq 2n-1$; (ii) $w_{\ge 1} \geq 2n-2$; and either (iii) $w_{\ge 2} \geq n-1$ or (iv) $w_{\ge 2} \geq n-2$ and $w_{\ge 3} \ge 1$.

We start with the following observation: $w_h \leq \binom{n}{h}$. Indeed, every subcube $s$ of weight $h$ contains the point $x_s$ obtained by switching all $*$s to $0$s, which has weight $h$. Since the subcubes in $F$ are disjoint, every $s \in F$ of weight $h$ has a different $x_s$. Since there are $\binom{n}{h}$ many possible $x_s$, it follows that $w_h \leq \binom{n}{h}$.

The inequality $w_{\ge 0} \geq 2n-1$ is \Cref{conj:minimal-subcubes}. Since $w_0 \leq 1$, the inequality $w_{\ge 1} \geq 2n-2$ follows. Since $w_1 \leq n$, we deduce the inequality $w_{\ge 2} \geq n-2$. To complete the proof, we need to show that either (iii) $w_{\ge 2} \geq n-1$ or (iv) $w_{\ge 3} \geq 1$. We will show that the assumptions $w_{\ge 2} = n-2$ and $w_{\ge 3} = 0$ lead to a contradiction.

\smallskip

Suppose, therefore, that $w_2 = n-2$ and $w_{\ge 3} = 0$. Since $w_{\ge 0} \geq 2n-1$ and $w_0 \leq 1$, $w_1 \leq n$, this implies that $w_0 = 1$ and $w_1 = n$.

Since $w_1 = n$, for every $i \in [n]$ there is a subcube $s^{(i)} \in F$ which contains $1$ in the $i$'th position: $s^{(i)}_i = 1$. The point $0^n$ is covered by the unique subcube $s^{(0)} \in F$ of weight $0$. Since $s^{(0)}$ and $s^{(i)}$ must conflict, necessarily $s^{(0)}_i = 0$ (this is the only possible conflict), and so $s^{(0)} = 0^n$ is a point.

Since $0^n \in F$, the subcubes $s^{(i)}$ cannot be points. Indeed, if $s^{(i)}$ is a point then $s^{(i)} = 0^{i-1}10^{n-i}$, and so $s^{(0)} \cup s^{(i)} = 0^{i-1}*0^{n-i}$, contradicting irreducibility. Consequently, all points in $F$ have even weight, contradicting \Cref{lem:parity} (applied on the star pattern $\emptyset$).
%
%Every subcube in $F$ which is not a point contains an equal number of points of even weight and of odd weight. In contrast, all points in $F$ have even weight. Since $\{0,1\}^n$ contains an equal number of points of either parity, we reach a contradiction.
\end{proof}

\subsubsection{Construction} \label{sec:weight-construction}

In this section, we prove (unconditionally) the upper bound part of \Cref{conj:minimal-weight}, by constructing tight irreducible subcube partitions of length $n \ge 3$ and weight vectors $1,n-1,n-1,0,\dots,0$ and $1,n,n-3,1,0,\dots,0$. The constructions will use the method of \Cref{thm:size-2n-1-merging}. The same subcube partitions can also be constructed using the method of \Cref{lem:twist}; we leave the details to the reader.

\begin{theorem} \label{thm:weight-1-n-1-n-1}
For each $n \ge 3$ there is a tight irreducible subcube partition $A_n$ whose weight vector is $1,n-1,n-1,0,\dots,0$.
\end{theorem}
\begin{proof}
We construct the subcube partitions inductively, starting with
\[
 A_3 = \{ *00, 001, 01*, 110, 1*1 \},
\]
which is obtained from $S_3$ of \Cref{thm:size-2n-1-merging} by flipping the third coordinate. The construction will maintain the invariants that $01*^{n-2} \in A_n$ and $1*^{n-1}, 00*^{n-2} \notin A_n$.

Given $A_n$, we construct $A_{n+1}$ by applying \Cref{lem:merge} to $F_0 = A_n$ and $F_1 = \{1*^{n-1}, 00*^{n-2}, 01*^{n-2}\}$, and rotating the result $G$ once to the left, that is, $A_{n+1} = \{ xb : bx \in G \}$, where $b \in \{0,1,*\}$ and $x \in \{0,1,*\}^n$.

Since $F_0$ is irreducible and $F_1$ is reducible, clearly $F_0 \neq F_1$, and so $A_{n+1}$ is tight by \Cref{lem:merge}.

The invariant implies that $F_1 \setminus F_0 = \{1*^{n-1}, 00*^{n-2}\}$ is irreducible, and so $A_{n+1}$ is irreducible by \Cref{lem:merge}.

The invariant states that $01*^{n-2} \in F_0$. Since $01*^{n-2} \in F_1$ by definition, it follows that $*01*^{n-2} \in G$, and so $01*^{n-1} \in A_{n+1}$. Since $A_{n+1}$ is irreducible, necessarily $00*^{n-1} \notin A_{n+1}$ (otherwise $00*^{n-1} \cup 01*^{n-1} = 0*^n$ would be a subcube) and $1*^n \notin A_{n+1}$ (otherwise the union of all other subcubes would be $0*^n$).

Finally, the invariants imply that $F_1 \setminus F_0 = \{1*^{n-1}, 00*^{n-2}\}$, and so compared to $F_0$, the subcube partition $G$ gains one subcube of weight $2$ (namely, $11*^{n-1}$) and one subcube of weight $1$ (namely, $100*^{n-2}$); all other subcubes originate from $F_0$ and maintain their weight.
\end{proof}

\begin{theorem} \label{thm:weight-1-n-n-3-1}
For each $n \ge 3$ there is a tight irreducible subcube partition $D_n$ whose weight vector is $1,n,n-3,1,0,\dots,0$.
\end{theorem}
\begin{proof}
The proof is very similar to that of \Cref{thm:weight-1-n-1-n-1}. We take
\[
 D_3 = S_3 = \{ *01, 000, 01*, 111, 1*0 \}.
\]
The rest of the proof is identical.
\end{proof}

The subcube partition $D_3$ is obtained from $A_3$ by flipping the third coordinate, and this holds for every $n$, by construction. By flipping coordinates appropriately, we can also obtain other tight irreducible subcube partitions $B_n,C_n$ whose weight vectors are $1,n-1,n-1,0,\dots,0$. The subcube partition $B_n$ is obtained by flipping the first coordinate, and $C_n$ is obtained by flipping both the first and the third coordinates. The subcube partitions $B_n,C_n$ can also be obtained using an iterative construction as above.

Here are the subcube partitions $A_5,B_5,C_5,D_5$ after rotation once to the left:
\[
\begin{array}{clclclc}
0000* & &0000* & &0100* & &0100* \\
01000 & &01001 & &00001 & &00000 \\
0*100 & &0*101 & &0*101 & &0*100 \\
0**10 & &0**11 & &0**11 & &0**10 \\
1***0 & &1***1 & &1***1 & &1***0 \\
10001 & &10000 & &11000 & &11001 \\
*1001 & &*1000 & &*0000 & &*0001 \\
**101 & &**100 & &**100 & &**101 \\
***11 & &***10 & &***10 & &***11 \\[0.5em]
A_5 && B_5 && C_5 && D_5
\end{array}
\]
Among all subcube partitions obtained from $A_n$ by flipping coordinates, these are the only ones whose weight vector is either $1,n-1,n-1,0,\dots,0$ or $1,n,n-3,1,0,\dots,0$.

\subsection{Maximal number of points}
\label{sec:maximal-points}

In the following section, we tackle the problem of maximizing the number of subcubes in an irreducible subcube partition. As a warm-up, we start with the problem of maximizing the number of points (zero-dimensional subcubes) in an irreducible subcube partition.

For $n = 3,4,5,6,7$, a computer search reveals that the maximum number of points in an irreducible subcube partition of length $n$ is $2,4,8,16,34$. An earlier version of the paper omitted the final value, leading us to conjecture that the maximum number of points in an irreducible subcube partition of length $n \ge 3$ is $2^{n-2}$. However, this is false for $n = 7,8,9$.

It is easy to see that an irreducible subcube partition of length $n$ contains at most $2^{n-1}$ points. Indeed, if the subcube partition contained more than $2^{n-1}$ points then there would be two points differing in a single coordinate. The union of these two points is an edge (a one-dimensional subcube), contradicting irreducibility.

\smallskip

In the rest of this section, we construct an irreducible subcube partition of length $n$ containing $2^{n-2}$ many points. The construction uses the following \namecref{lem:xor-subst-2}.

\begin{lemma} \label{lem:xor-subst-2}
If $F \neq \{0*^{n-1}, 1*^{n-1}\}$ is an irreducible subcube partition of length $n$ then the following is an irreducible subcube partition of length $n + 1$:
\[
 G = \{ 00t, 11t : 0t \in F \} \cup \{ 01t, 10t : 1t \in F \} \cup \{ **t : *t \in F \}.
\]
\end{lemma}
\begin{proof}
Let $F'$ be the irreducible subcube partition obtained from $F$ by flipping the first coordinate. The subcube partition $G$ results from applying \Cref{lem:merge} to $F$ and $F'$. According to the \namecref{lem:merge}, in order to show that $G$ is irreducible, it suffices to show that $F \cap F' \neq \emptyset$.

If $F \cap F' = \emptyset$ then all subcubes in $F$ start with $0$ or $1$. Hence the union of all subcubes in $F$ starting with $0$ is $0*^{n-1}$, and the union of all subcubes in $F$ starting with $1$ is $1*^{n-1}$. Since $F$ is irreducible, it follows that $F = \{0*^{n-1}, 1*^{n-1}\}$, contrary to the assumption. Therefore $F \cap F' \neq \emptyset$, completing the proof.
\end{proof}

\begin{corollary} \label{cor:xor-subst}
If $F \neq \{0*^{n-1}, 1*^{n-1}\}$ is an irreducible subcube partition of length $n$ then for every $k \ge 1$, the following is an irreducible subcube partition of length $n + k$:
\[
 G = \{a_1\ldots a_k t : bt \in F, a_1,\dots,a_k,b \in \{0,1\}, a_1 \oplus \dots \oplus a_k = b\} \cup \{ *^k t : *t \in F \}.
\]
\end{corollary}
\begin{proof}
Apply the \namecref{lem:xor-subst-2} iteratively $k-1$ times to $F$, noticing that the subcube partition constructed in the \namecref{lem:xor-subst-2} is never of the form $\{ 0*^m, 1*^m \}$.
\end{proof}

In order to construct an irreducible subcube partition of length $n$ with $2^{n-2}$ many points, we apply the \namecref{cor:xor-subst} to the irreducible subcube partition $S_3$ from \Cref{thm:size-2n-1-merging}.

\begin{theorem} \label{thm:many-points}
For every $n \ge 3$ there is an irreducible subcube partition of length $n$ containing $2^{n-2}$ many points.
\end{theorem}
\begin{proof}
The subcube partition $S_3 = \{ 000, *01, 1*0, 01*, 111 \}$ is irreducible (according to \Cref{thm:size-2n-1-merging}) and contains two points. Applying \Cref{cor:xor-subst} with $k = n-3$, we get an irreducible subcube partition of length $n$ in which each of the two original points gives rise to $2^{n-3}$ points, for a total of $2^{n-2}$ points.
\end{proof}

Using similar ideas, we can construct an irreducible subcube partition of length $n$ containing any even number of points between $2$ and $2^{n-2}$. We leave the details to the reader.

\subsection{Maximal size}
\label{sec:maximal-subcubes}

What is the maximal size of an irreducible subcube partition of length $n$? Here are some values, based on experiments and an upper bound which we present in \Cref{sec:maximal-subcubes-ub}:

\[
\begin{array}{r|ccccccc}
n & 3 & 4 & 5 & 6 & 7 & 8 & 9 \\\hline
\text{Lower bound} & 5 & 9 & 20 & 40 & 81 & 162 & 324 \\
\text{Upper bound} & 5 & 9 & 20 & 40 & 81 & 166 & 334
\end{array}
\]

An earlier version of the paper presented inferior bounds which led us to conjecture that the maximal size of an irreducible subcube partition of length $n \ge 5$ is $5 \cdot 2^{n-3}$. However, this is false for $n = 7$.

We give a construction of an irreducible subcube partition of size $5 \cdot 2^{n-3}$ for every $n \ge 5$ in \Cref{sec:maximal-subcubes-construction}. We give a better construction for $n = 7$ in \Cref{sec:maximal-subcubes-construction-better}.

\subsubsection{Upper bound} \label{sec:maximal-subcubes-ub}

In this section, we use a result of Forcade~\cite{Forcade73} to give an upper bound on the size of irreducible subcube partitions.

\begin{theorem} \label{thm:maximal-subcubes-ub}
For every $n \ge 3$, the size of any irreducible subcube partition of length $n$ is at most
\[
 \frac{2n-1}{3n-1} 2^n = \left(\frac{16}{3} - \Theta\left(\frac{1}{n}\right) \right) 2^{n-3}.
\]
\end{theorem}
\begin{proof}
Let $F$ be an irreducible subcube partition of length $n$. 
Let $G$ be a subcube partition obtained from $F$ by subdividing each subcube of dimension larger than $1$ into edges (subcubes of dimension $1$) in an arbitrary way.

Since $F$ is irreducible, no two points in $G$ span an edge. Therefore the set of edges in $G$ constitutes a maximal matching in the $n$-dimensional hypercube. Forcade~\cite{Forcade73} proved that any maximal matching in the $n$-dimensional hypercube contains $m \geq \frac{n}{3n-1} 2^n$ edges. Therefore
\[
 |F| \leq |G| = m + (2^n - 2m) = 2^n - m \leq 2^n - \frac{n}{3n-1} 2^n. \qedhere
\]
\end{proof}

Forcade showed that the bound $\frac{n}{3n-1} 2^n$ is asymptotically tight by giving a matching construction, and so the bound obtained in \Cref{thm:maximal-subcubes-ub} is optimal for this proof technique. %Therefore this method cannot prove the conjectured upper bound $5 \cdot 2^{n-3}$.

When $n = 7$, a computer search finds that the maximum number of points is $34$, and so the maximum size is at most $34 + \frac{128-34}{2} = 81$.

\subsubsection{General construction} \label{sec:maximal-subcubes-construction}

In this section we construct irreducible subcube partitions of size $5 \cdot 2^{n-3}$ for all $n \ge 3$ except for $n = 4$. When $n = 4$, a computer search reveals that the maximum number of subcubes is $9$, which is achieved by
\[
 0000, 0011, 1101, 1110, *100, *111, 0*01, 0*10, 10**.
\]
We prove this below.

Our construction is based on the work of Perezhogin~\cite{Perezhogin05}, brought to our attention by Tarannikov~\cite{Tarannikov23}.%
\footnote{The construction on page 55 in \cite{Perezhogin05} contains several minor errors. Tarannikov pointed out to us that the definition of the $M_{\sigma_i \sigma_j}^{n+2}$ there should read
\begin{align*}
  M_{00}^{n+2} = (M_{00}^n(00) \cup M_{01}^n(01) \cup M_{11}^n(11) \cup M_{10}^n(10) \cup \mathbf{v}_1^{00} \cup \mathbf{v}_2^{00} \cup \mathbf{v}_3^{00} \cup \mathbf{v}_4^{00}) \setminus (\mathbf{e}_1(00) \cup \mathbf{e}_2(01) \cup \mathbf{e}_3(11) \cup \mathbf{e}_4(10)),\\
  M_{01}^{n+2} = (M_{00}^n(01) \cup M_{01}^n(11) \cup M_{11}^n(10) \cup M_{10}^n(00) \cup \mathbf{v}_1^{01} \cup \mathbf{v}_2^{10} \cup \mathbf{v}_3^{01} \cup \mathbf{v}_4^{10}) \setminus (\mathbf{e}_1(01) \cup \mathbf{e}_2(11) \cup \mathbf{e}_3(10) \cup \mathbf{e}_4(00)),\\
  M_{11}^{n+2} = (M_{00}^n(11) \cup M_{01}^n(10) \cup M_{11}^n(00) \cup M_{10}^n(01) \cup \mathbf{v}_1^{11} \cup \mathbf{v}_2^{11} \cup \mathbf{v}_3^{11} \cup \mathbf{v}_4^{11}) \setminus (\mathbf{e}_1(11) \cup \mathbf{e}_2(10) \cup \mathbf{e}_3(00) \cup \mathbf{e}_4(01)),\\
  M_{10}^{n+2} = (M_{00}^n(10) \cup M_{01}^n(00) \cup M_{11}^n(01) \cup M_{10}^n(11) \cup \mathbf{v}_1^{10} \cup \mathbf{v}_2^{01} \cup \mathbf{v}_3^{10} \cup \mathbf{v}_4^{01}) \setminus (\mathbf{e}_1(10) \cup \mathbf{e}_2(00) \cup \mathbf{e}_3(01) \cup \mathbf{e}_4(11)).
\end{align*}
}
The construction is inductive, increasing the length by~$2$ at each step. Consequently, we will need two base cases, for $n = 3$ and for $n = 6$. The same inductive construction will also be used in \Cref{sec:homogeneous-partitions} to construct subcube partitions consisting only of edges for all $n \ge 4$ except for $n = 5$.

\smallskip

We start with the inductive step. The construction uses a mapping from $\ZZf$ to $\{0,1\}^2$:
\[
\begin{array}{cccc}
0 & 1 & 2 & 3 \\\hline
00 & 01 & 11 & 10
\end{array}
\]
From a geometric perspective, we enumerate vertices of a two-dimensional face clockwise:
%%%%%%
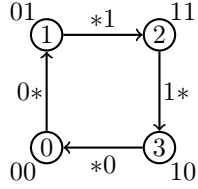
\begin{center}
\tikzset{
    vert/.style = {align=center, inner sep=0pt, text centered, circle, draw, text width=1em, thick,
    every edge quotes/.style = {auto}}
}
\begin{tikzpicture}[vert=[circle, draw]],
\foreach \i/\x/\y/\ps in {0/0/0/{south west}, 1/0/1/{north west}, 2/1/1/{north east}, 3/1/0/{south east}}
{
    \node[vert,label={\ps:\x\y}] (a\x\y) at (1.5*\x,1.5*\y) {\i};
};
\foreach \a/\b/\q in {00/01/0*, 01/11/*1, 11/10/1*, 10/00/*0}
\draw[->] (a\a) edge["$\q$"] (a\b);
\end{tikzpicture}
\end{center}
%%%%%%%
We denote this mapping, known as the \emph{Gray map}, by $\alpha \mapsto \enc{\alpha}$. We use the same notation to associate two adjacent elements of $\ZZf$ with the corresponding edge, e.g.\ $\enc{1,2} = *1$.

Let $\WWf$ consist of all subsets of $\ZZf$ which are either singletons or pairs of adjacent elements. 
The function $\phi \mapsto \enc{\phi}$ maps $\WWf$ into $\{0,1,*\}^2 \setminus \{**\}$.
Given $\phi \in \WWf$ and $\alpha \in \ZZf$, we define $\phi + \alpha$ in the obvious way. For example, $\{1\} + 2 = \{3\}$ and $\{1,2\} + 2 = \{3,0\} = \{0,3\}$.

The construction will apply to subcube partitions which satisfy the following \emph{complementation property}.

\begin{definition}[Complementation property]
A subcube partition $F$ of length $n \ge 2$ satisfies the \emph{complementation property} if the following properties hold:
\begin{enumerate}[(a)]
\item No subcube in $F$ ends with $**$.
\item If $s\enc{\phi} \in F$, where $\phi \in \WWf$, then $s\enc{\phi+2} \notin F$.
\end{enumerate}
(Note that $\enc{\alpha+2}$ is obtained from $\enc{\alpha}$ by complementing both symbols, leaving stars untouched.)
\end{definition}

The construction is given by the following lemma.

\begin{lemma} \label{lem:perezhogin}
Let $F$ be an irreducible subcube partition of length $n \ge 2$ containing $0^{n-2}\enc{0,1}$ and satisfying the complementation property. Define
\[
 F' = \{ s \enc{\phi-\alpha} \enc{\alpha} : s\enc{\phi} \in F, \alpha \in \ZZf, s\enc{\alpha} \neq 0^{n-2}\enc{0,1} \} \cup \{ 0^{n-2} \enc{-\alpha} \enc{\alpha,\alpha+1} : \alpha \in \ZZf \}.
\]
Then $F'$ is a tight irreducible subcube partition of length $n + 2$ containing $0^n\enc{0,1}$ and satisfying the complementation property. Moreover, for every $d$, if $F$ contains $m$ subcubes of dimension $d$ then $F'$ contains $4m$ subcubes of dimension $d$.
\end{lemma}
\begin{proof}

We can write $F'$ as follows:
\[
 F' = \underbrace{\{ s \enc{\phi-\alpha} \enc{\alpha} : s\enc{\phi} \in F\}}_{F''} \setminus
 \underbrace{\{ 0^{n-2} \enc{\alpha,\alpha+1} \enc{-\alpha} \}}_{\text{removed edges}} \cup
 \underbrace{\{ 0^{n-2} \enc{-\alpha} \enc{\alpha,\alpha+1} \}}_{\text{added edges}},
\]
where $\alpha$ ranges over $\ZZf$. We use the terms ``removed edges'' and ``added edges'' below to refer to the sets in the expression above.

Most of the properties (apart from irreducibility) are easy to verify:
\begin{description}
\item[Subcube partition] The removed edges and the added edges cover the same set of points, namely
\[
 \{ 0^{n-2} \enc{\alpha} \enc{\beta} : \alpha + \beta \in \{0,1\} \},
\]
and so $F'$ is a subcube partition.

\item[Tightness] The added edges $0^{n-2}\enc{0}\enc{0,1}$ and $0^{n-2}\enc{3}\enc{1,2}$ together mention all coordinates.
\item[Contains $0^n\enc{0,1}$] This is one of the added edges.
\item[Complementation property] By construction, no subcube in $F'$ ends with $**$. It remains to show that if $t\enc{\psi} \in F'$ then $t\enc{\psi+2} \notin F'$. If $\psi$ is a pair then $t\enc{\psi}$ is one of the added edges, and each of these has a different $t$. \\ If $\psi = \{\alpha\}$ is a singleton then $t\enc{\psi} = s \enc{\phi-\alpha} \enc{\alpha}$ for some $s \enc{\phi} \in F$. If $F'$ also contains $t \enc{\psi + 2} = s \enc{\phi-\alpha} \enc{\alpha + 2}$ then $s \enc{\phi+2} \in F$, contradicting the assumption that $F$ satisfies the complementation property.
\item[Subcube counts] Since $\enc{\phi + \alpha}$ has the same number of stars as $\enc{\phi}$ and $\enc{\alpha}$ has no stars, each subcube in $F$ gives rise to four subcubes of the same dimension in $F''$. Getting from $F''$ to $F'$ involved adding and removing four edges, hence the claim about subcube counts.
\end{description}

\paragraph{Irreducibility}

The main part of the proof is proving the irreducibility of $F'$. Let $G'$ be a non-empty set whose union is a subcube. We need to show that either $|G'| = 1$ or $G' = F'$.

For every $\alpha \in \ZZf$, define
\begin{align*}
 G_\alpha &= \{ s\enc{\phi} : s\enc{\phi - \alpha} \enc{\alpha} \in G' \}, \\
 H_\alpha &= \{ s\enc{\phi - \alpha} : s\enc{\phi - \alpha} \enc{\alpha} \in G' \}, \\
 P_\alpha &= \{ t \in \{0,1\}^n : t \enc{\alpha} \in \bigcup G' \}.
\end{align*}
By construction, $G_\alpha \subseteq F \setminus \{0^{n-2}\enc{0,1}\}$.

We consider several cases, according to which added edges (if any) belong to $G'$.

\paragraph{No added edges} Suppose first that $G'$ contains no added edge. In this case,
\[
 G' = \bigcup_\alpha \{ t \enc{\alpha} : t \in H_\alpha \}.
\]
This implies that $\bigcup H_\alpha$, and so $\bigcup G_\alpha$, are subcubes of $F$ for each $\alpha$. By irreducibility of $F$, each $G_\alpha$ is either empty or a singleton.

If there is a single non-empty $H_\alpha$ then $G'$ is a singleton, and we are done. Suppose therefore that at least two $H_\beta$ are non-empty.
If $H_\alpha,H_{\alpha+2}$ are non-empty then since $\enc{\alpha} \lor \enc{\alpha+2} = **$, in fact $H_\beta$ is non-empty for all $\beta$. Therefore we can assume that $H_\alpha,H_{\alpha+1}$ are non-empty for some $\alpha$.

Suppose $H_\alpha = \{ s \enc{\phi-\alpha} \}$, where $s\enc{\phi} \in F$. Since $\bigcup G'$ is a subcube, we have $\bigcup H_\alpha = \bigcup H_{\alpha+1}$. Recalling that $H_{\alpha+1}$ is a singleton, this shows that $s \enc{\phi-\alpha} \in H_{\alpha+1}$, and so $s\enc{\phi+1} \in F$.
Thus $s\enc{\phi},s\enc{\phi+1} \in F$.

If $\phi=\{\alpha\}$ then $s\enc{\alpha} \cup s\enc{\alpha+1} = s\enc{\alpha,\alpha+1}$, contradicting the irreducibility of $F$. If $\phi = \{\alpha,\alpha+1\}$, then $s\enc{\phi} \cap s\enc{\phi+1} = s\enc{\alpha+1}$, contradicting the assumption that $F$ is a partition. Therefore this case cannot happen.

\paragraph{Single added edge} Suppose next that $G'$ contains a single added edge, $0^{n-2} \enc{-\alpha} \enc{\alpha,\alpha+1}$. Then
\begin{align*}
P_\alpha &= \bigcup H_\alpha \cup \{0^{n-2} \enc{-\alpha} \}, \\
\bigcup G_\alpha &= \{ s \enc{\beta+\alpha} : s \enc{\beta} \in P_\alpha, s\enc{\beta} \neq 0^{n-2} \enc{-\alpha} \} \\ &=
\{ s \enc{\beta+\alpha} : s \enc{\beta} \in P_\alpha \} \setminus \{0^{n-2} \enc{0} \}, \\
P_{\alpha+1} &= \bigcup H_{\alpha+1} \cup \{0^{n-2} \enc{-\alpha}\}, \\
\bigcup G_{\alpha+1} &= \{ s \enc{\beta+\alpha+1} : s \enc{\beta} \in P_{\alpha+1}, s\enc{\beta} \neq 0^{n-2} \enc{-\alpha} \} \\ &=
\{ s \enc{\beta+\alpha+1} : s \enc{\beta} \in P_{\alpha+1} \} \setminus \{0^{n-2} \enc{1}\}.
\end{align*}

Since $\bigcup G'$ is a subcube, we have $P_\alpha = P_{\alpha+1}$, and the common value is a subcube containing $0^{n-2} \enc{-\alpha}$. We claim that this subcube must be of the form $r \enc{-\alpha}$. This is because of the following:
\begin{itemize}
\item Since $0^{n-2}\enc{0,1} \in F$, in particular $0^{n-2}\enc{1} \notin F$, and so $0^{n-2}\enc{1-\alpha} \notin P_\alpha$.
\item By assumption $G'$ doesn't contain the added edge $0^{n-2} \enc{-\alpha-1} \enc{\alpha+1,\alpha+2}$, and so $0^{n-2}\enc{-\alpha-1} \notin P_{\alpha+1} = P_\alpha$.
\item Since $\enc{-\alpha} \lor \enc{2-\alpha} = **$, also $0^{n-2} \enc{2-\alpha} \notin P_\alpha$ (otherwise $P_\alpha$ would contain $0^{n-2}**$ since $P_\alpha$ is a subcube).
\end{itemize}

Notice now that
\[
 \bigcup (G_\alpha \cup G_{\alpha+1} \cup 0^{n-2} \enc{0,1}) =
 \{ s \enc{\beta+\alpha}, s \enc{\beta+\alpha+1} : s \enc{\beta} \in P_\alpha \} = r\enc{0,1}.
\]
The irreducibility of $F$ implies that $r\enc{0,1} = 0^{n-2} \enc{0,1}$. Therefore either $\bigcup G' = 0^{n-2} \enc{-\alpha} \enc{\alpha,\alpha+1}$ or $\bigcup G' = 0^{n-2} \enc{-\alpha} **$. In the former case, $G'$ is a singleton consisting just of a single added edge, and so it remains to rule out the latter case.

If $\bigcup G' = 0^{n-2} \enc{-\alpha} **$ then in particular $0^{n-2} \enc{-\alpha} \enc{\alpha+2,\alpha+3} \subseteq \bigcup G'$, and so $\bigcup H_{\alpha+2} = 0^{n-2} \enc{-\alpha}$ and $\bigcup H_{\alpha+3} = 0^{n-2} \enc{-\alpha}$. This implies that $\bigcup G_{\alpha+2} = 0^{n-2} \enc{2}$ and $\bigcup G_{\alpha+3} = 0^{n-2} \enc{3}$, and so both $0^{n-2} \enc{2},0^{n-2} \enc{3} \in F$. However, this contradicts the irreducibility of $F$, since $0^{n-2}\enc{2} \cup 0^{n-2} \enc{3} = 0^{n-2} \enc{2,3}$ is a subcube.

\paragraph{Exactly two adjacent added edges} Suppose now that $G'$ contains precisely two added edges, whose indices differ by $1$, say $0^{n-2}\enc{-\alpha}\enc{\alpha,\alpha+1}$ and $0^{n-2}\enc{-\alpha-1}\enc{\alpha+1,\alpha+2}$. Therefore
\[
 P_{\alpha+1} = \bigcup H_{\alpha+1} \cup \{0^{n-2} \enc{-\alpha-1,-\alpha}\},
\]
which implies that
\[
 \bigcup G_{\alpha+1} = \{ s \enc{\beta+\alpha+1} : s\enc{\beta} \in P_{\alpha+1} \} \setminus \{0^{n-2} \enc{0,1}\}.
\]

Since $\bigcup G'$ is a subcube, also $P_{\alpha+1}$ is a subcube, and so $\bigcup (G_{\alpha+1} \cup 0^{n-2} \enc{0,1})$ is a subcube. The irreducibility of $F$ implies that $G_{\alpha+1} = \emptyset$, and so $P_{\alpha+1} = 0^{n-2} \enc{-\alpha-1,-\alpha}$.

The join of the two added edges is $0^{n-2}\enc{-\alpha-1,-\alpha}**$, and this shows that all $P_\beta$ are equal.
In particular, since $P_\alpha = \bigcup H_\alpha \cup \{0^{n-2} \enc{-\alpha}\}$, we have $H_\alpha = \{0^{n-2} \enc{-\alpha-1}\}$ and so $G_\alpha = \{0^{n-2} \enc{3}\}$. Similarly, $P_{\alpha+2} = \bigcup H_{\alpha+2} \cup \{0^{n-2} \enc{-\alpha-1}\}$ and so $H_{\alpha+2} = \{0^{n-2} \enc{-\alpha}\}$, implying that $G_{\alpha+2} = \{0^{n-2} \enc{2}\}$. However, then $0^{n-2}\enc{2},0^{n-2}\enc{3} \in F$, which contradicts the irreducibility of $F$ since $0^{n-2}\enc{2} \cup 0^{n-2}\enc{3} = 0^{n-2}\enc{2,3}$ is a subcube.

\paragraph{Two non-adjacent added edges} In the remaining case, the set $G'$ contains two non-adjacent edges, say $0^{n-2} \enc{-\alpha} \enc{\alpha,\alpha+1}$ and $0^{n-2} \enc{-\alpha-2} \enc{\alpha+2,\alpha+3}$. The join of these two edges is $0^{n-2} ****$, and so $\bigcup G'$ is a subcube ending with $****$. In particular, $G'$ must contain all added edges.

Let $G''$ be obtained from $G'$ by replacing the added edges with the removed edges. Since the added edges and the removed edges span the same points, $\bigcup G''$ is still a subcube. Consider
\[
 G''_0 = \{ s \enc{\phi} : s \enc{\phi} \enc{0} \in G'' \} \subseteq F.
\]
Since $\bigcup G''_0$ is a subcube, the irreducibility of $F$ implies that either $G''_0$ is a singleton or $\bigcup G''_0 = \{0,1\}^n$. In the latter case, clearly $\bigcup G' = \{0,1\}^{n+2}$, and so $G' = F'$. In the former case, since the removed edge $0^{n-2}\enc{0,1}$ belongs to $G''_0$ by construction, we see that $G''_0 = \{0^{n-2} \enc{0,1}\}$.
However, since $\bigcup G''$ is a subcube ending with $****$, we see that $\bigcup G''_0$ should be a subcube ending with $**$ rather than with $\enc{0,1}$. So this case cannot happen.
\end{proof}

We can now construct the desired subcube partitions.

\begin{theorem} \label{thm:maximal-construction}
For every $n \ge 3$ other than $n = 4$ there is a tight irreducible subcube partition of length $n$ and size $5 \cdot 2^{n-3}$.
\end{theorem}
\begin{proof}
Using \Cref{lem:perezhogin}, it suffices to construct tight irreducible subcube partitions of lengths $n \in \{3,6\}$ containing $0^{n-1}*$ and satisfying the complementation property in the lemma.

For $n = 3$, we can take
\[
 100, *10, 1*1, 00*, 011.
\]

For $n = 6$, we can take
\begin{align*}
&\mathit{abcdef}\colon a \oplus b = c \oplus f = \overline{d \oplus e} \\
&*aba\bar{b}\bar{a} \\
&a*\bar{a}b\bar{a}b \\
&ab*\bar{b}\bar{a}b \\
&ab\bar{a}*ab \\
&aba\bar{b}*\bar{b} \\
&ababa*
\end{align*}
Here $a,b,c,d,e,f$ range over $\{0,1\}$.
\end{proof}

We close this section by showing that $n = 4$ is indeed exceptional. We first need the following lemma, which will also be useful in \Cref{sec:homogeneous-partitions}.

\begin{lemma} \label{lem:length-4}
Let $F$ be an irreducible subcube partition of length $4$. The set of points in $F$ is either empty or one of the following, up to permutation and flipping of coordinates:
\begin{align*}
&0000 & &0000 \\
&1110 & &1110 \\
&     & &1101 \\
&     & &0011
\end{align*}
\end{lemma}
\begin{proof}
Suppose that $F$ contains a point, say $0000 \in F$.
According to \Cref{lem:parity}, $F$ must contain a point of odd parity. Due to irreducibility, this point needs to have weight $3$, say $1110 \in F$.

If $|F| > 2$ then according to \Cref{lem:parity}, $F$ contains another point of odd parity, which due to irreducibility needs to have weight $3$, say $1101 \in F$. According to \Cref{lem:parity}, $F$ needs to contain another point of even parity, which due to irreducibility cannot be one of $1100,1010,0110,1001,0101,1111$. Therefore $0011 \in F$.

Due to irreducibility, $1011,0111 \notin F$, and so \Cref{lem:parity} shows that in this case $F$ contains precisely the following points: $0000,1110,1101,0011$.
\end{proof}

We can now show that \Cref{thm:maximal-construction} cannot hold for $n = 4$.

\begin{lemma} \label{lem:maximal-4}
There is no irreducible subcube partition of length $4$ and size $10$.
\end{lemma}
\begin{proof}
Let $F$ be an irreducible subcube partition of length $4$ and size $10$. According to \Cref{lem:length-4}, $F$ contains at most $4$ points, and so has to consist of $4$ points and $6$ edges. Moreover, without loss of generality the points in $F$ are $0000, 1110, 1101, 0011$.

Let $F_i$ consist of all subcubes $s \in F$ with $s_i = *$. Note that $F_i$ is non-empty, since otherwise the union of subcubes $s \in F$ with $s_i = 0$ is $*^{i-1}0*^{4-i}$, contradicting irreducibility.

There are exactly two points in $F$ belonging to $*^{i-1}0*^{4-i}$. Each edge outside of $F_i$ contains an even number of such points, and each edge in $F_i$ contains one such point. Therefore $|F_i|$ is even, and so $|F_i| \ge 2$. This implies that $|F_1| + |F_2| + |F_3| + |F_4| \geq 8$, contradicting the fact that $F$ contains only $6$ edges.
\end{proof}

This lemma also follows from the computation of Peitl and Szeider~\cite[Table~1]{PeitlSzeider22}.

\subsubsection{Improved construction for small \texorpdfstring{$n$}{n}}
\label{sec:maximal-subcubes-construction-better}

In this section we exhibit irreducible subcube partitions of size exceeding $5 \cdot 2^{n-3}$ for $n=7,8,9$. These partitions consist only of points and edges. To save space, we will only list the \emph{edges}; everything not covered by the edges is a point. The irreducibility of these subcube partitions can be verified using \Cref{alg:reducibility}.

An irreducible subcube partition for $n = 7$ consists of $34$ points and $47$ edges:

\[
\begin{array}{cccccccc}
*000010 & *001011 & *010000 & *011111 & *100000 & *111011 & *111110 & 0*00101 \\ 0*11101 & 00*1001 & 00000*1 & 00001*0 & 0001*00 & 000111* & 001*100 & 001001* \\ 00101*1 & 0011*10 & 01*0001 & 01*0110 & 010*010 & 010*111 & 01010*1 & 010110* \\ 0110*00 & 0110*11 & 01110*0 & 1*01101 & 1*10110 & 1*11000 & 10*1010 & 100*001 \\ 1000*00 & 100011* & 10011*0 & 101*011 & 1010*01 & 101110* & 11*0010 & 11*1111 \\ 1100*11 & 110010* & 1101*10 & 110100* & 111*001 & 111*100 & 11101*1 & 
\end{array}
\]

An irreducible subcube partition for $n = 8$ consists of $68$ points and $94$ edges:

\small{
\[
\begin{array}{cccccccccc}
*0001100 & *0011110 & *0100011 & *0101110 & *0111111 & *1000000 & *1010001 & *1011100 & *1100001 & *1110011 \\
0*010010 & 0*010100 & 0*011000 & 0*011101 & 0*100010 & 0*100111 & 0*101011 & 0*101101 & 00*00100 & 00*10011 \\
00*10101 & 00*11001 & 000*0000 & 000*0110 & 0000*001 & 00000*11 & 0000101* & 000011*1 & 00011*11 & 00100*01 \\
001010*0 & 0011000* & 0011011* & 00111*10 & 0011110* & 01*00110 & 01*01010 & 01*01100 & 01*11011 & 01000*01 \\
0100001* & 0100100* & 0100111* & 010101*1 & 01011*10 & 011*1001 & 011*1111 & 01100*00 & 0111*110 & 011100*0 \\
0111010* & 01111*00 & 1*000100 & 1*000111 & 1*001110 & 1*110001 & 1*111000 & 1*111011 & 10*10010 & 10*10100 \\
10*10111 & 10000*01 & 10000*10 & 10001*11 & 1000100* & 1001*011 & 1001000* & 100110*0 & 1001110* & 101*0101 \\
101*0110 & 101*1001 & 101*1010 & 101*1100 & 101000*0 & 101011*1 & 11*01000 & 11*01011 & 11*01101 & 110*0011 \\
110*0101 & 110*0110 & 110*1001 & 110*1010 & 110100*0 & 110111*1 & 1110*100 & 1110001* & 111001*1 & 1110111* \\
11110*00 & 1111011* & 11111*01 & 11111*10
\end{array}
\]
}

An irreducible subcube partition for $n = 9$ consists of $136$ points and $188$ edges:

\small{
\[
\begin{array}{ccccccccc}
*00011110 & *00100100 & *00111111 & *01000000 & *01011011 & *01100001 & *10000110 & *10100111 & *10111100 \\
*11000011 & *11011000 & *11111001 & 0*0000111 & 0*0010000 & 0*0010101 & 0*0100011 & 0*0101000 & 0*0101101 \\
0*0110010 & 0*1001101 & 0*1010010 & 0*1010111 & 0*1011100 & 0*1101010 & 0*1101111 & 0*1111000 & 000*00000 \\
000*01010 & 000*10001 & 000*11101 & 0000*0010 & 00000*100 & 0000000*1 & 000001*01 & 00000111* & 00001*011 \\
00001011* & 000011*00 & 00010*110 & 0001001*1 & 0001010*1 & 00011*100 & 000110*11 & 000111*10 & 00011100* \\
001*00010 & 001*01110 & 001*10101 & 001*11111 & 00100*011 & 001000*01 & 00100011* & 001001*00 & 00101*001 \\
0010101*0 & 0010110*0 & 0011*1101 & 00110*100 & 001100*11 & 00110100* & 00111*011 & 001110*10 & 00111000* \\
0011111*0 & 01*000101 & 01*001111 & 01*010011 & 01*010100 & 01*011101 & 01*100010 & 01*101011 & 01*101100 \\
01*110000 & 01*111010 & 010*10110 & 010*11111 & 01000*010 & 01000000* & 010001*00 & 0100010*1 & 010011*10 \\
01001100* & 0101*0001 & 0101*0100 & 01010111* & 0101101*1 & 0101110*1 & 011*00000 & 011*01001 & 0110*1011 \\
0110*1110 & 0110001*0 & 0110010*0 & 01101000* & 011100*01 & 01110011* & 01111*101 & 011110*11 & 0111101*0 \\
01111111* & 1*0001000 & 1*0001101 & 1*0011111 & 1*0101010 & 1*0110000 & 1*0110101 & 1*0111011 & 1*1000100 \\
1*1001010 & 1*1001111 & 1*1010101 & 1*1100000 & 1*1110010 & 1*1110111 & 10*000101 & 10*001011 & 10*001100 \\
10*010111 & 10*011101 & 10*100010 & 10*101000 & 10*110011 & 10*110100 & 10*111010 & 100*00111 & 100*01110 \\
100000*10 & 10000000* & 10001*010 & 100010*00 & 1000100*1 & 10001100* & 1001*1001 & 1001*1100 & 1001000*1 \\
1001011*1 & 10011011* & 101*10001 & 101*11000 & 1010*0011 & 1010*0110 & 10100100* & 1010100*0 & 1010111*0 \\
10110*101 & 10110011* & 101101*11 & 1011011*0 & 101111*01 & 10111111* & 110*00101 & 110*01001 & 110*10010 \\
110*11000 & 1100*1010 & 11000*011 & 110000*00 & 11000111* & 11001*100 & 110010*01 & 11001011* & 1100110*1 \\
11010*100 & 110100*10 & 11010000* & 110101*11 & 11011*110 & 1101100*1 & 1101111*1 & 111*00111 & 111*01101 \\
111*10110 & 111*11010 & 11100*001 & 1110000*0 & 1110011*0 & 11101*011 & 111010*00 & 111011*01 & 11101111* \\
1111*0101 & 11110*011 & 1111001*0 & 111101*10 & 11110100* & 11111*100 & 11111000* & 111111*11
\end{array}
\]
}

\subsection{Maximal minimum dimension}
\label{sec:maximal-min-dim}

All irreducible subcube partitions we have exhibited so far contain points. Is this necessary? More generally, given $n$, what is the maximal $d$ such that there exists a tight irreducible subcube partition in which every subcube has dimension at least $d$? (The question doesn't make sense without assuming tightness, since $*^n$ is always irreducible.)

The constructions we give below suggest the following conjecture.

\begin{definition}[Minimum dimension] \label{def:min-dim}
For a subcube partition $F$, let $\delta(F)$ denote the minimum dimension of a subcube of $F$.

Also, let $\delta_*(F)$ denote the minimum dimension of a subcube of $F$ ending with a star (if such a subcube exists), and let $\delta_b(F)$ denote the minimum dimension of a subcube of $F$ not ending with a star (if such a subcube exists).
\end{definition}

\begin{conjecture} \label{conj:min-dim}
Every tight irreducible subcube partition $F$ of length $n$ satisfies $\delta(F) \leq n/2 - o(n)$.
\end{conjecture}

One can similarly ask for the maximum value of $\Delta(F)$, which is the minimum codimension of a subcube of $F$, over all irreducible subcube partitions of length $n$. The results in \Cref{sec:maximal-subcubes} shows that when $n \geq 3$ the maximal value is $\Delta(F) = n-1$ (the case $n = 4$ is covered in \Cref{sec:homogeneous-partitions}).

\smallskip

In the remainder of this section, we give a construction matching \Cref{conj:min-dim}. The construction is based on the following \namecref{lem:merge-stars}.

\begin{lemma} \label{lem:merge-stars}
Let $F$ be a subcube partition of length $n$. Define
\[
 G = \{ *t** : t* \in F \} \cup \{ 0tb*, 1t*b : b \in \{0,1\}, tb \in F  \}.
\]
Then
\begin{enumerate}[(a)]
\item $G$ is a subcube partition of length $n + 2$.
\item If $F$ is tight then $G$ is tight.
\item If $F$ is irreducible and contains a subcube ending with a star then $G$ is irreducible and contains a subcube ending with a star.
\item We have $\delta_*(G) = \min(\delta_*(F) + 2, \delta_b(F) + 1)$ and $\delta_b(G) = \delta_b(F) + 1$, and so $\delta(G) \ge \delta(F) + 1$. Moreover, if $\delta(F) = \delta_b(F)$ then $\delta(G) = \delta_*(G) = \delta_b(G) = \delta(F) + 1$.
%\item $\delta(G) \geq \delta(F) + 1$.
%Furthermore, if $\delta(F)$ is attained at a subcube not ending with a star then the same holds for $G$, and moreover $\delta(G) = \delta(F) + 1$. 
\end{enumerate}
\end{lemma}
\begin{proof}
Let $F_0 = \{ tc* : tc \in F \}$ and $F_1 = \{ t*c : tc \in F \}$, where $c \in \{0,1,*\}$ in both cases. Applying \Cref{lem:merge} to $F_0,F_1$, we obtain the subcube partition $G$.

Suppose that $F$ is tight. For every $i \in \{1,\dots,n-1\}$, some subcube of $F$ mentions coordinate $i$. The corresponding subcube or subcubes of $G$ mention coordinate $i+1$. Some subcube of $F$ mentions coordinate $n$. The corresponding subcubes of $G$ mention the remaining coordinates $1,n+1,n+2$.

Suppose that $F$ is irreducible and contains a subcube $s \in F$ ending with a star. The irreducibility of $F$ directly implies the irreducibility of $F_0$ and $F_1$. Since $s* \in F_0 \cap F_1$, \Cref{lem:merge} implies that $G$ is irreducible. Furthermore, $*s* \in G$ is a subcube ending with a star.

The remaining claims are easy to verify directly once we notice that the dimension of a subcube is the number of star coordinates.
\end{proof}

We apply the construction on three specific tight irreducible subcube partitions (one only for $n = 4$) in order to obtain the following result, which gives the best constructions we are aware of.

\begin{theorem} \label{thm:min-dim}
For every odd $n \ge 3$ there is a tight irreducible subcube partition $F$ of length $n$ with $\delta(F) = \frac{n-3}{2}$.

For $n = 4$ there is a tight irreducible subcube partition $F$ of length $n$ with $\delta(F) = \frac{n-4}{2}$.

For every even $n \ge 6$ there is a tight irreducible subcube partition $F$ of length $n$ with $\delta(F) = \frac{n-2}{2}$.
\end{theorem}
\begin{proof}
The first part follows from applying \Cref{lem:merge-stars} to the tight irreducible subcube partition $S_3$ of \Cref{thm:size-2n-1-merging}. The second part follows from taking the tight irreducible subcube partition $S_4$ of the same \namecref{thm:size-2n-1-merging}. The third part follows from applying \Cref{lem:merge-stars} to the following tight irreducible subcube partition, whose irreducibility can be checked using \Cref{alg:reducibility}:
\begin{align*}
&0*0*1* & &00**0* & &001*1* & &010*0* & &0110** \\
&1**0*1 & &10***0 & &10*1*1 & &11*0*0 & &1101** \\
&&&&&&&&&*111** \qedhere
\end{align*}
% \begin{align*}
% &0*0*1* \\
% &1**0*1 \\
% &00**0* \\
% &10***0 \\
% &001*1* \\
% &10*1*1 \\
% &010*0* \\
% &11*0*0 \\
% &*111** \\
% &0110** \\
% &1101** \qedhere
% \end{align*}
\end{proof}

\subsection{Homogeneous subcube partitions}
\label{sec:homogeneous-partitions}

So far we have considered various parameters of irreducible subcube partitions, attempting to optimize them. The final question we consider concerns subcube partitions in which all subcubes have the same codimension.

\begin{definition}[Homogeneity] \label{def:homogeneous}
An \emph{($n,k)$-homogeneous} subcube partition is a tight subcube partition of length $n$ in which all subcubes have codimension $k$.
\end{definition}

In this section, we explore the following question: for which $n,k$ does there exist an irreducible $(n,k)$-homogeneous subcube partition?

Here is a table with some experimental results:
\[
\begin{array}{c|cccccc}
& n=4 & n=5 & n=6 & n=7 & n=8 & n=9 \\\hline
k=3 & \taby & \tabn & \tabn & \tabn & \tabn & \tabn \\
k=4 & & \tabn & \taby & \tabn & \tabn & \tabn \\
k=5 & & & \taby & \taby & \taby & \taby \\
k=6 & & & & \taby & \taby & \taby \\
k=7 & & & & & \taby & \taby \\
k=8 & & & & & & \taby
\end{array}
\]

In \Cref{sec:homogeneous-elementary} we prove several elementary results: an irreducible $(n,1)$-homogeneous subcube partition exists only for $n = 1$; no irreducible $(n,2)$-homogeneous partition exists; and for $k \ge 3$, if an irreducible $(n,k)$-homogeneous subcube partition exists then $k+1 \leq n \leq 2^k-3$.
In \Cref{sec:homogeneous-weights} we show that the weight distribution of an $(n,k)$-homogeneous subcube partition is binomial.

In \Cref{sec:homogeneous-perezhogin}, we describe a construction of Perezhogin~\cite{Perezhogin05}, which gives irreducible $(n,n-1)$-homogeneous subcube partitions for all $n \ge 4$ other than $n = 5$. We also show that no irreducible $(5,4)$-homogeneous subcube partition exists, and that an irreducible $(6,4)$-homogeneous subcube partition does exist.

In \Cref{sec:homogeneous-construction}, we show how the irreducible subcube partitions constructed in \Cref{sec:homogeneous-perezhogin} give rise to even more irreducible subcube partitions, using a simple inductive construction.

Finally, in \Cref{sec:homogeneous-3} we show that an irreducible $(n,3)$-homogeneous partition exists only for $n = 4$, and in \Cref{sec:homogeneous-4} we show that an irreducible $(n,4)$-homogeneous partition exists only for $n = 6$ (with the help of a computer).

\subsubsection{Elementary bounds} \label{sec:homogeneous-elementary}

We start with the following general bound.

\begin{lemma} \label{lem:homogeneous-k-2k}
Suppose that $n \ge 4$ and $k \ge 2$. If there exists an irreducible $(n,k)$-homogeneous subcube partition then $k+1 \leq n \leq 2^k - 3$.
\end{lemma}
\begin{proof}
Let $F$ be an irreducible $(n,k)$-homogeneous subcube partition. Clearly $n \geq k$. If $n = k$ then all subcubes in $F$ are points, and so $F$ is not irreducible. Hence $n \ge k+1$. Since $F$ has size $2^k$, the upper bound $n \leq 2^k - 3$ follows from \Cref{thm:size-lower-bound}.
\end{proof}

The following corollary of \Cref{lem:parity} will be useful. The corollary itself, and its applications below, were suggested to us by Kisielewicz~\cite{Kisielewicz23}.

\begin{corollary} \label{cor:parity-homogeneous}
Let $F$ be a homogeneous subcube partition, and let $S$ be a star pattern occurring in $F$. Among subcubes in $F$ whose star pattern is $S$, half have even parity and half have odd parity.
\end{corollary}
\begin{proof}
Since $F$ is homogeneous, all star patterns in $F$ are inclusion-minimal, and so the \namecref{cor:parity-homogeneous} follows immediately from \Cref{lem:parity} (applied with $G = F$).
\end{proof}

We now determine when an irreducible $(n,k)$-homogeneous subcube partition exists for $k = 1$ and $k = 2$.

\begin{lemma} \label{lem:homogeneous-1}
If $F$ is an irreducible $(n,1)$-homogeneous subcube partition then $n = 1$ and $F = \{0, 1\}$.
\end{lemma}
\begin{proof}
The two subcubes in $F$ contain a single non-star position, which must be identical. Since $F$ is tight, necessarily $n = 1$, and so $F = \{0, 1\}$.
\end{proof}

\begin{lemma} \label{lem:homogeneous-2}
There are no irreducible $(n,2)$-homogeneous subcube partitions, for any $n$.
\end{lemma}
\begin{proof}
Let $F$ be an $(n,2)$-homogeneous subcube partition.
Suppose, without loss of generality, that $00*^{n-2} \in F$. \Cref{cor:parity-homogeneous} implies that $F$ must contain a subcube with the same star pattern and odd parity, without loss of generality $01*^{n-2}$. Since $00*^{n-2} \cup 01*^{n-2} = 0*^{n-1}$ is a subcube, $F$ is reducible.
\end{proof}

Here is an alternative proof, suggested by the reviewer. As in the proof of \Cref{lem:homogeneous-k-2k}, we have $n \ge k + 1 = 3$. On the other hand, the result of Tarsi mentioned in the proof of \Cref{thm:size-lower-bound} implies that $n \leq 2^k - 1 = 3$. Hence $n = 3$, and this case can be ruled out by hand (and also follows from~\cite[Lemma~41]{KullmannZhao16}).

\subsubsection{Weight distribution} \label{sec:homogeneous-weights}

In this section we prove the following surprising property, which involves the concept of weight vector defined in \Cref{sec:minimal-weight}. 

\begin{lemma} \label{lem:homogeneous-weight}
The weight vector of any $(n,k)$-homogeneous subcube partition is
\[
 \binom{k}{0},\binom{k}{1},\dots,\binom{k}{k},0,\dots,0.
\]
\end{lemma}
\begin{proof}
Let $F$ be an $(n,k)$-homogeneous subcube partition, and let $w$ be its weight vector. Considering the number of points of weight $\ell$ which are covered, for each $\ell \in \{0,\ldots,k\}$ we have
\[
 \sum_{r=0}^\ell \binom{n-k}{\ell-r} w_r = \binom{n}{\ell}.
\]
This is a triangular system of equations, and so it has a unique solution. In other words, all $(n,k)$-homogeneous subcube partitions (if any) have the same weight vector.

The argument above applies even if we don't assume that $F$ is tight. Therefore all $(n,k)$-homogeneous subcube partitions have the same weight vector as the subcube partition $\{ x*^{n-k} : x \in \{0,1\}^k \}$, whose weight vector is the one in the statement of the \namecref{lem:homogeneous-weight}.
\end{proof}

\subsubsection{Special perfect matchings} \label{sec:homogeneous-perezhogin}

Perezhogin~\cite{Perezhogin05} defines a \emph{special perfect matching} to be a perfect matching in the hypercube graph which is irreducible (in our terminology). If the hypercube has dimension $n$, then this is the same as an irreducible $(n,n-1)$-homogeneous subcube partition. He constructs a special perfect matching for all $n \ge 4$ other than $n = 5$, and shows that no special perfect matching exists when $n = 5$. In this section, we give an exposition of his work.

We start with the construction.

\begin{theorem} \label{thm:special-perfect-matchings}
For every $n \ge 4$ other than $n = 5$ there exists an irreducible $(n,n-1)$-homogeneous subcube partition.
\end{theorem}
\begin{proof}
Using \Cref{lem:perezhogin}, it suffices to construct irreducible $(n,n-1)$-homogeneous subcube partitions of lengths $n \in \{4,7\}$ containing $0^{n-1}*$ and satisfying the complementation property in the lemma.

For $n = 4$, we can take
\[ 0*10, 01*1, 000*, 1*01, 10*0, 111*, *100, *011. \]

For $n = 7$, we can take the subcube partition obtained by applying \Cref{lem:merge} with $F_0$ being the subcube partition constructed in \Cref{thm:maximal-construction} for $n = 6$, and with $F_1$ being obtained from $F_0$ by flipping the first two coordinates.
\end{proof}

The following result shows that \Cref{thm:special-perfect-matchings} cannot be extended to $n = 5$.

\begin{lemma} \label{lem:homogeneous-5-4}
There is no irreducible $(5,4)$-homogeneous subcube partition.
\end{lemma}
\begin{proof}
Suppose that $F$ is an irreducible $(5,4)$-homogeneous subcube partition. Every subcube in $F$ contains a single $*$. For $i \in \{1,\dots,5\}$, let $F_i \subset F$ consist of those subcubes $s \in F$ with $s_i = *$. According to \Cref{lem:parity}, $|F_i|$ is even. If $|F_i| = 0$ then the union of all subcubes $s \in F$ with $s_i = 0$ is the subcube $*^{i-1}0*^{5-i}$, contradicting irreducibility, and so $|F_i| > 0$.

The proof of \Cref{lem:length-4} implies that for each $i \in \{1,\dots,5\}$, either $|F_i| = 2$ or $|F_i| = 4$. Moreover, if $|F_i| = 4$ then each coordinate in $F_i$ is \emph{balanced} ($0$ appears the same number of times as $1$), whereas if $|F_i| = 2$ then exactly one coordinate is unbalanced. In the entire formula, each coordinate is balanced, since the union of subcubes with $s_i = 0$ is the same as the union of subcubes with $s_i = 1$, and each subcube contains exactly two points.

\smallskip

Without loss of generality, $|F_1| \le \cdots \le |F_5|$. Since 
$|F_1| + \cdots + |F_5| = 16$, it follows that $|F_1| = |F_2| = 2$ and $|F_3| = |F_4| = |F_5| = 4$.
Recall that each of $F_1$ and $F_2$ contains exactly one unbalanced coordinates. Since all coordinates in $F_3,F_4,F_5,F$ are balanced, $F_1$ and $F_2$ have the same unbalanced coordinate. In particular, the second coordinate in $F_1$ and the first coordinate in $F_2$ are balanced.

Let $F_1 = \{*0x_0, *1x_1\}$ and $F_2 = \{0*y_0,1*y_1\}$. For $\alpha,\beta \in \{0,1\}$, we have $\{ s : \alpha\beta s \in F_1 \cup F_2 \} = \{x_\alpha,y_\beta\}$. The set $\{ s : \alpha\beta s \in F_3 \cup F_4 \cup F_5 \}$ consists of three edges, each covering a point of even parity and a point of odd parity. Therefore $x_\alpha$ and $y_\beta$ have opposite parity.

Suppose that $x_\alpha$ and $y_\beta$ differ in a single bit, say $x_\alpha = 000$ and $y_\beta = 001$. The set $\{ s : \alpha\beta s \in F_3 \cup F_4 \cup F_5 \}$ must be a matching of the remaining points. There are only four possibilities:
\begin{align*}
&01*,10*,11* \\
&01*,1*0,1*1 \\
&*10,*01,*11 \\
&*10,0*1,1*1
\end{align*}
In each of them, $F$ is reducible. Therefore $y_\beta$ must be the negation of $x_\alpha$. However, this implies that $x_0 = x_1$, contradicting irreducibility.
\end{proof}

In contrast, there does exist an irreducible $(6,4)$-homogeneous subcube partition.

\begin{lemma} \label{lem:homogeneous-6-4}
There exists an irreducible $(6,4)$-homogeneous subcube partition.
\end{lemma}
\begin{proof}
Here is such a subcube partition:
\begin{gather*}
0000**, 001**1, 01*01*, 01**00, 0*01*1, 0**110, *010*0, *0*100, \\
1101**, 111**0, 10*11*, 10**01, 1*00*0, 1**011, *111*1, *1*001. \qedhere
\end{gather*}
\end{proof}

Computer search reveals that up to permutation and flipping of coordinates, the irreducible $(6,4)$-homogeneous subcube partition is unique.

\subsubsection{More infinite families} \label{sec:homogeneous-construction}

In this section we show how any irreducible homogeneous subcube partition gives rise to an infinite family.

\begin{lemma} \label{lem:homogeneous-pump}
Let $k \ge 2$. If there exists an irreducible $(n,k)$-homogeneous subcube partition then there exists an irreducible $(3n,2k)$-homogeneous subcube partition.
\end{lemma}
\begin{proof}
Let $F$ be an irreducible $(n,k)$-homogeneous subcube partition. We start by observing that for each coordinate $i$, there must be some subcube $s \in F$ with $s_i = *$. Otherwise, the union of the subcubes $s \in F$ with $s_i = 0$ will be $*^{i-1} 0 *^{n-i}$, which contradicts irreducibility.

Repeat the following operation $n$ times to $F$: apply \Cref{lem:merge-stars}, and rotate the result twice to the right (equivalently, replace $t* \in F$ with $***t$ and $tb \in F$ with $b*0t,*b1t$, where $b \neq *$). The observation in the preceding paragraph ensures that the resulting subcube partition $G$ is tight and irreducible. By construction, $G$ has length $n + 2n = 3n$ and codimension $k + k = 2k$.
\end{proof}

Applying this to the results of \Cref{sec:homogeneous-perezhogin}, we obtain the following infinite families.

\begin{corollary} \label{cor:homogeneous-infinite}
For every $t \geq 0$ and every $n \geq 4$ other than $n = 5$ there exists an irreducible $(3^t \cdot n, 2^t \cdot (n-1))$-homogeneous subcube partition.

For every $t \ge 0$ there exists an irreducible $(3^t \cdot 6, 2^t \cdot 4)$-homogeneous subcube partition.
\end{corollary}

\subsubsection{Codimension 3} \label{sec:homogeneous-3}

\Cref{thm:special-perfect-matchings} shows that an irreducible $(4,3)$-homogeneous subcube partition exists. In this section, we show that an irreducible $(n,3)$-homogeneous subcube partition exists only for $n = 4$, and that it is unique up to permutation and flipping of coordinates.

\begin{theorem} \label{thm:homogeneous-3}
If there exists an irreducible $(n,3)$-homogeneous subcube partition then $n = 4$. Moreover, the irreducible $(4,3)$-homogeneous subcube partition is unique up to permutation and flipping of coordinates.
\end{theorem}
\begin{proof}
Let $F$ be an irreducible $(n,3)$-homogeneous subcube partition.
We start by proving the following claims:
\begin{enumerate}[(i)]
\item If $s \in F$ then $\bar{s} \in F$, where $\bar{s}$ is obtained by flipping all bits in $s$. \label{itm:homo-3-neg}
\item If $s \in F$ then $s$ and $\bar{s}$ are the only subcubes in $F$ of the star pattern $P(s)$. \label{itm:homo-3-pair}
\item Every two subcubes in $F$ have at least two non-star coordinates in common. \label{itm:homo-3-int}
\end{enumerate}

Suppose that $s \in F$, without loss of generality $s = 000*^{n-3}$. \Cref{cor:parity-homogeneous} implies that $F$ contains another subcube having the same star pattern but with opposite parity. Since $F$ is irreducible, this cannot be one of $100*^{n-3},010*^{n-3},001*^{n-3}$, and so $111*^{n-3} \in F$, proving \Cref{itm:homo-3-neg}. Irreducibility implies that no other subcube with the same star pattern can belong to $F$, proving \Cref{itm:homo-3-pair}. Any other subcube in $F$ must conflict with both $000*^{n-3}$ and $111*^{n-3}$, hence must mention at least two coordinates among the first three, proving \Cref{itm:homo-3-int}.

\smallskip

Suppose now without loss of generality that $000*^{n-3} \in F$. According to \Cref{itm:homo-3-neg}, also $111*^{n-3} \in F$. According to \Cref{itm:homo-3-pair}, these are the only two subcubes with this star pattern. Hence $F$ must contain a subcube $s$ with a different star pattern. According to \Cref{itm:homo-3-int}, the subcube $s$ must mention at least two coordinates out of $\{1,2,3\}$. Since $s$ must conflict with both $000*^{n-3}$ and $111*^{n-3}$, without loss of generality $s = 01*0*^{n-4} \in F$. According to \Cref{itm:homo-3-neg}, also $10*1*^{n-4} \in F$, and according to \Cref{itm:homo-3-pair}, these are the only subcubes in $F$ with this star pattern.

According to \Cref{itm:homo-3-int}, every other subcube $t$ in $F$ must mention two coordinates out of $\{1,2,3\}$ and two coordinates out of $\{1,2,4\}$. If $t$ mentions coordinates $1,2$ (and a third coordinate not in $\{3,4\}$) then $t$ cannot possibly conflict with all of $000*^{n-3},111*^{n-3},01*0*^{n-4},10*1*^{n-4}$, and so $t$ must mention either $\{1,3,4\}$ or $\{2,3,4\}$. Since $F$ is tight, we deduce that $n = 4$.

According to \Cref{itm:homo-3-pair}, $F$ contains precisely two subcubes mentioning $\{1,3,4\}$ and precisely two subcubes mentioning $\{2,3,4\}$. Moreover, in each pair, one of the subcubes is the negation of the other, according to \Cref{itm:homo-3-neg}. 
The subcube of the form $0*ab$ must be $0*11$ in order to conflict with $000*$ and $01*0$, and so $0*11,1*00 \in F$. Similarly, the subcube of the form $*0ab$ must be $*010$ in order to conflict with $000*$ and $10*0$, and so $*010,*101 \in F$.
Since $F$ contains eight subcubes, this completes the description of $F$, and so $F$ is unique up to permutation and flipping of coordinates.
\end{proof}

\subsubsection{Codimension 4} \label{sec:homogeneous-4}

\Cref{lem:homogeneous-6-4} shows that an irreducible $(6,4)$-homogeneous subcube partition exists. Using techniques similar to the preceding section, in this section
we show that an irreducible $(n,4)$-homogeneous subcube partition exists only for $n = 6$.

\begin{theorem} \label{thm:homogeneous-4}
If there exists an irreducible $(n,4)$-homogeneous subcube partition then $n = 6$.
%Moreover, the irreducible $(6,4)$-homogeneous subcube partition is unique up to permutation and flipping of coordinates.
\end{theorem}

Since the proof is a bit long, we break it into three parts, starting with the following \namecref{lem:homogeneous-4-aux}.

\begin{lemma} \label{lem:homogeneous-4-aux}
If $F$ is an irreducible $(n,4)$-homogeneous subcube partition and $s,t \in F$ are two different subcubes, then $s,t$ have at least two non-star coordinates in common.
\end{lemma}
\begin{proof}
Let $F$ be an irreducible $(n,4)$-homogeneous subcube partition. We start with the following observations:
\begin{enumerate}[(i)]
\item If $u \in F$ then $F$ contains another subcube with the same star pattern differing in exactly three coordinates. \label{itm:homo-4-3}
\item For every coordinate $i$, there is some $v \in F$ with $v_i = *$. \label{itm:homo-4-star}
\end{enumerate}

Let $u \in F$, say $u = 0000*^{n-4}$. According to \Cref{cor:parity-homogeneous}, $F$ must contain another subcube with the same star pattern and opposite parity. Since $F$ is irreducible, this cannot be one of $1000*^{n-4},\allowbreak 0100*^{n-4},\allowbreak 0010*^{n-4},\allowbreak 0001*^{n-4}$, implying \Cref{itm:homo-4-3}.

If all $v \in F$ satisfy $v_i \in \{0,1\}$ then the union of all subcubes with $v_i = 0$ is the subcube $*^{i-1}0*^{n-i}$, hence $*^{i-1}0*^{n-i} \in F$ by irreducibility; but then $F$ is not $(n,4)$-homogeneous. This proves \Cref{itm:homo-4-star}.

\smallskip

Suppose, for the sake of contradiction, that $F$ contains two subcubes $s,t$ which share fewer than two non-star coordinates. Since $s,t$ conflict, they have exactly one non-star coordinate in common. Without loss of generality, $s=0000*^{n-4}$ and $t=***1000*^{n-7}$. Applying \Cref{itm:homo-4-3}, $F$ must contain one of $0111*^{n-4},\allowbreak 1011*^{n-4},\allowbreak 1101*^{n-4},\allowbreak 1110*^{n-4}$. The only one of these which conflicts with $t$ is $1110*^{n-4}$. Applying \Cref{itm:homo-4-3} to $t$, we similarly get that $***1111 *^{n-7} \in F$. Thus $F$ contains the following subcubes:
\begin{align*}
&0000*** \, *^{n-7} \\
&1110*** \, *^{n-7} \\
&***1000 \, *^{n-7} \\
&***1111 \, *^{n-7}
\end{align*}
According to \Cref{itm:homo-4-star} with $i = 4$, the subcube partition $F$ must contain a subcube $v$ with $v_4 = *$. In order to conflict with the first two subcubes above, $v$ must contain a $0$ and a $1$ in coordinates $1,2,3$. In order to conflict with the latter two subcubes, it must contain a $0$ and a $1$ in coordinates $5,6,7$. Without loss of generality, $v$ is the following subcube:
\begin{align*}
&01**01* \, *^{n-7}
\end{align*}
Applying \Cref{itm:homo-4-3} to $v$, one of the following subcubes belongs to $F$:
\begin{align*}
&00**10* \, *^{n-7} \\
&11**10* \, *^{n-7} \\
&10**00* \, *^{n-7} \\
&10**11* \, *^{n-7}
\end{align*}
However, the $i$'th subcube in this list fails to conflict with the $i$'th subcube in the previous list, and we reach the desired contradiction.
\end{proof}

We use \Cref{lem:homogeneous-4-aux} together with the following \namecref{lem:homogeneous-6-times} to bound $n$.

\begin{lemma} \label{lem:homogeneous-6-times}
Let $F$ be an irreducible $(n,k)$-homogeneous subcube partition, where $n \ge 2$. Each coordinate is mentioned in at least six subcubes of $F$.
\end{lemma}
\begin{proof}
We will show that the first coordinate is mentioned at least six times. For $\sigma \in \{0,1,*\}$, let $F_\sigma = \{ x : \sigma x \in F \}$. Since $F_0 \cup F_*$ and $F_1 \cup F_*$ are both subcube partitions, $\bigcup F_0 = \bigcup F_1$, and so $|F_0| = |F_1|$.

\Cref{lem:irreducible-regular} shows that $|F_0| \geq 2$. If $|F_0| = 2$ then let $F_0 = \{s,t\}$ and $F_1 = \{s',t'\}$. Notice that $s \cup t = s' \cup t'$ and the union is not a subcube, since otherwise $F$ would contain $0s,0t$ whose union is a subcube, contradicting irreducibility. \Cref{lem:nfs-pair} shows that either $\{s,t\} = \{s',t'\}$ or $s,t$ are an nfs-pair (in some order). In the former case, $F$ contains $0s,1s$, contradicting irreducibility. The latter case is impossible by homogeneity, since the two subcubes in an nfs-pair have different dimensions.
\end{proof}

We can now prove the \namecref{thm:homogeneous-4}.

\begin{proof}[Proof of \Cref{thm:homogeneous-4}]
Let $F$ be an irreducible $(n,4)$-homogeneous subcube partition. Suppose without loss of generality that $0000*^{n-4} \in F$. According to \Cref{lem:homogeneous-4-aux}, every other subcube in $F$ mentions at most two coordinates beyond the first four, and so at most $2 \cdot 15 / 6 = 5$ of these are mentioned at least six times. \Cref{lem:homogeneous-6-times} implies that $n \leq 4 + 5 = 9$.

We can slightly improve on this, as follows. Let $u^{(1)},u^{(2)},u^{(3)},u^{(4)}$ be the subcubes containing the points $1000\,0^{n-4},0100\,0^{n-4},0010\,0^{n-4},0001\,0^{n-4}$, respectively. Each of these subcubes must be different. Indeed, if for example $u^{(1)} = u^{(2)}$ then $u^{(1)} \supseteq 1000\,0^{n-4} \lor 0100\,0^{n-4} = **00\,0^{n-4}$, which intersects with $0000*^{n-4}$.

Any two of $u^{(1)},u^{(2)},u^{(3)},u^{(4)}$ must conflict, and so for distinct $i,j \in \{1,\dots,4\}$, either $u^{(i)}_j = 0$ or $u^{(j)}_i = 0$. This means that together, $u^{(1)},u^{(2)},u^{(3)},u^{(4)}$ contain at least $\binom{4}{2} = 6$ zeroes among the first four coordinates. Therefore one of $u^{(1)},u^{(2)},u^{(3)},u^{(4)}$ must contain at least $\lceil 6/4 \rceil = 2$ zeroes among the first two coordinates, and so mentions at most one coordinate beyond the first four.

This means that strictly fewer than $2 \cdot 15 / 6 = 5$ coordinates are mentioned at least six times, and so $n \leq 4 + 4 = 8$.

Recalling that $n \geq 5$ due to \Cref{lem:homogeneous-k-2k}, we complete the proof of the theorem by checking with a computer that no irreducible $(n,4)$-homogeneous subcube partitions exist for $n=5,7,8$.
%and that the irreducible $(6,4)$-homogeneous subcube partition is unique up to permutation and flipping of coordinates.
(The case $n = 5$ was also worked out by hand in \Cref{lem:homogeneous-5-4}.)
\end{proof}

\section{Nonbinary subcube partitions}
\label{sec:larger-alphabets}

So far we have considered subcube partitions of the hypercube $\{0,1\}^n$. In this section, we study subcube partitions of $\{0,\dots,q-1\}^n$ for arbitrary $q \ge 2$.

\begin{definition}[Subcube partition] \label{def:q-subcube-partition}
A \emph{subcube partition} of $\{0,\dots,q-1\}^n$ (or: a subcube partition \emph{over} $\{0,\dots,q-1\}$ of length $n$) is a partition of $\{0,\dots,q-1\}^n$ into \emph{subcubes}, which are sets of the form
\[
 \{ x \in \{0,\dots,q-1\}^n : x_{i_1} = b_1, \dots, x_{i_d} = b_d \}.
\]
\end{definition}

We identify subcubes with words over $\{0,\dots,q-1,*\}$. The definitions of the following concepts are identical to the binary case: dimension and codimension of a subcube, point, edge, size (\Cref{def:subcube-partition}); reducible subcube partition (\Cref{def:reducible}); tight subcube partition (\Cref{def:tight}); conflicting subcubes (\Cref{def:conflicting-subcubes}).

Given a collection $F$ of subcubes of $\{0,\dots,q-1\}^n$, we can determine whether they form a subcube partition using the criterion of \Cref{lem:testing-partition}, replacing $2$ with $q$. Determining whether a subcube partition of $\{0,\dots,q-1\}^n$ is tight is easy using the definition, and we can determine irreducibility using  \Cref{alg:reducibility}.

\smallskip

We start our exploration of subcube partitions over $\{0,\dots,q-1\}$ in \Cref{sec:expansion}, where we show how to convert an irreducible subcube partition of $\{0,1\}^n$ into an irreducible subcube partition of $\{0,\dots,q-1\}^n$.

We then study the minimal size of tight irreducible subcube partitions over $\{0,\dots,q-1\}$ in \Cref{sec:larger-alphabets-minimal-subcubes}. 

\subsection{Expansion}
\label{sec:expansion}

In this section we show how to convert a subcube partition of $\{0,1\}^n$ into a subcube partition of $\{0,\dots,q-1\}^n$ in a way which preserves tightness and irreducibility.

\begin{lemma} \label{lem:expansion}
Let $F$ be a subcube partition of $\{0,1\}^n$, let $q \ge 2$, and let $\phi_1,\dots,\phi_n\colon \{0,\dots,q-1\} \to \{0,1\}$ be surjective functions.

Extend the definitions of $\phi_1,\dots,\phi_n$ to $\{0,\dots,q-1,*\}$ by defining $\phi_i(*) = *$. Define a function $\phi\colon \{0,\dots,q-1,*\}^n \to \{0,1,*\}^n$ as follows: $\phi(\sigma_1 \ldots \sigma_n) = \phi_1(\sigma_1) \ldots \phi_n(\sigma_n)$. Let
\[
 G = \{ s \in \{0,\dots,q-1,*\}^n : \phi(s) \in F \}.
\]
Then
\begin{enumerate}[(a)]
\item $G$ is a subcube partition of $\{0,\dots,q-1\}^n$.
\item If $F$ is tight then so is $G$.
\item If $F$ is irreducible then so is $G$.
\end{enumerate}
\end{lemma}
\begin{proof}
We start by showing that $G$ is a subcube partition. Notice first that the subcubes in $G$ are disjoint. Indeed, suppose that $s,s' \in G$ are distinct. If $\phi(s) = \phi(s')$ then $s,s'$ must disagree on a non-star position, and so conflict. If $\phi(s) \neq \phi(s')$ then $\phi(s),\phi(s')$ conflict at some position $i$, and $s,s'$ conflict at the same position.

In order to show that the subcubes in $G$ cover all of $\{0,\dots,q-1\}^n$, let $x \in \{0,\dots,q-1\}^n$. Since $F$ is a subcube partition, $\phi(x)$ is covered by some subcube $t \in F$. Define a subcube $s$ as follows: if $t_i = *$ then $s_i = *$, and otherwise $s_i = x_i$. Then $\phi(s) = t$ and so $s \in G$, and $s$ covers $x$ by definition.

\smallskip

Now suppose that $F$ is tight. Then for every $i \in [n]$ there is a subcube $t \in F$ mentioning $i$. Since $\phi$ is surjective, we can find a subcube $s$ mentioning $i$ such that $\phi(s) = t$. Hence $s \in G$, and so $G$ also contains a subcube mentioning $i$. Hence $G$ is tight.

\smallskip

Finally, suppose that $F$ is irreducible. If $G$ is reducible then there is a a subset $H \subset G$, with $1 < |H| < |G|$, whose union is a subcube $r$. We claim that the union of $\phi(H) = \{ \phi(s) : s \in H \}$ is the subcube $\phi(r)$.

Indeed, on the one hand, any $s \in H$ satisfies $s \subseteq r$ and so $\phi(s) \subseteq \phi(r)$, hence $\bigcup \phi(H) \subseteq \phi(r)$. On the other hand, let $x \in \phi(r)$ be an arbitrary point. Define a point $y \in \{0,\dots,q-1\}^n$ as follows: if $r_i = *$ then $y_i$ is an arbitrary element of $\phi_i^{-1}(x_i)$, and otherwise $y_i = r_i$; in the latter case, $\phi_i(y_i) = \phi_i(r_i) = x_i$. By construction, $y \in r$, and so $y$ is covered by some $s \in H$. Since $\phi(y) = x$, it follows that $\phi(s)$ covers $x$.

Since $F$ is irreducible, either $|\phi(H)| = 1$ or $|\phi(H)| = |F|$. In the latter case, $\phi(r) = *^n$ and so $r = *^n$, implying that $H = G$, contrary to assumption.
In the former case, $\phi(s) = \phi(r)$ for all $s \in H$. Choose two distinct subcubes $s,s' \in H$. Let $i \in [n]$ be a coordinate at which $s,s'$ conflict. Since $r \supseteq s \lor s'$ we have $r_i = *$, and so $\phi(r)_i = *$. On the other hand, $s_i,s'_i \neq *$, contradicting $\phi(s) = \phi(s') = \phi(r)$.
\end{proof}

\subsection{Minimal size}
\label{sec:larger-alphabets-minimal-subcubes}

\Cref{sec:minimal-subcubes} studies the minimal size of a tight irreducible subcube partition of $\{0,1\}^n$. In this section we extend this study to tight irreducible subcube partitions of $\{0,\dots,q-1\}^n$, asking: what is the minimal size of a tight irreducible subcube partition of $\{0,\dots,q-1\}^n$?

Applying \Cref{lem:expansion} to the tight irreducible subcube partitions constructed in \Cref{thm:weight-1-n-1-n-1}, we obtain a tight irreducible subcube partition of size $(n-1)q(q-1) + 1$. We conjecture that this is optimal.

\begin{conjecture} \label{conj:minimal-size-q}
If $n \ge 3$ then for all $q \ge 2$, the minimal size of a tight irreducible subcube partition of $\{0,\dots,q-1\}^n$ is $(n-1)q(q-1) + 1$.
\end{conjecture}

We formally describe the matching construction in \Cref{sec:larger-alphabets-construction}, where we also show that this is the minimal size that can be achieved by a direct application of \Cref{lem:expansion}, assuming \Cref{conj:minimal-subcubes}.

We prove \Cref{conj:minimal-size-q} for $n = 3$ in \Cref{sec:larger-alphabets-small-n}, where we also show that no tight irreducible subcube partition exists for $n = 2$. We have also verified the conjecture using a computer for $n = 4$ and $q \leq 6$, as well as for $n = 5$ and $q = 3$.

We close the section by proving a modest lower bound of $(q-1)n + 1$ on the size of a tight subcube partition of $\{0,\dots,q-1\}^n$, using the technique of Tarsi~\cite{AharoniLinial86}. The lower bound applies more generally to tight minimal subcube covers, where it is sharp.

\subsubsection{Construction}
\label{sec:larger-alphabets-construction}

In this section we show how to construct tight irreducible subcube partitions of $\{0,\dots,q-1\}^n$ of size $(n-1)q(q-1) + 1$ using \Cref{lem:expansion}, and explain why this is the minimal possible size when using the \namecref{lem:expansion}, assuming \Cref{conj:minimal-subcubes}. We start with the construction.

\begin{theorem} \label{thm:larger-alphabets-construction}
For each $n \ge 3$ and $q \ge 2$ there exists a tight irreducible subcube partition of $\{0,\dots,q-1\}^n$ of size $(n-1)q(q-1) + 1$.
\end{theorem}
\begin{proof}
\Cref{thm:weight-1-n-1-n-1} constructs a tight irreducible subcube partition of $\{0,1\}^n$ whose weight vector is $1,n-1,n-1,0,\dots,0$. Applying \Cref{lem:expansion} with the mappings $\phi_i$ given by $\phi_i(0) = 0$ and $\phi_i(1) = \cdots = \phi_i(q-1) = 1$ for all $i \in [n]$, we obtain a tight irreducible subcube partition of size
\[
 1 \cdot (q-1)^0 + (n-1) \cdot (q-1)^1 + (n-1) \cdot (q-1)^2 = (n-1)q(q-1) + 1. \qedhere
\]
\end{proof}

We now show that this construction is the optimal way of applying \Cref{lem:expansion}, assuming \Cref{conj:minimal-subcubes}.

\begin{theorem} \label{thm:larger-alphabets-expansion-optimal}
Assume that \Cref{conj:minimal-subcubes} holds for some $n \ge 3$. Let $F$ be a tight irreducible subcube partition of $\{0,1\}^n$. Let $G$ be a subcube partition obtained by an application of \Cref{lem:expansion} on $F$, for some $q \ge 2$. Then $G$ has size at least $(n-1)q(q-1) + 1$.
\end{theorem}
\begin{proof}
Let $g(z_1,\dots,z_n)$ be the size of $G$ when \Cref{lem:expansion} is applied with functions $\phi_1,\dots,\phi_n\colon \{0,\dots,q-1\} \to \{0,1\}$ such that $|\phi_i^{-1}(0)| = z_i$ for all $i \in [n]$. The function $g$ is multilinear, and so its minimal value over $\{1,\dots,q-1\}^n$ is attained at some $z \in \{1,q-1\}^n$.
Define a subcube partition $F'$ by flipping all coordinates $i$ such that $z_i = q-1$. Then
\[
 |G| \geq g(z) = \sum_{s \in F'} (q-1)^{\nones{s}}.
\]
Since $F'$ is tight and irreducible, a combination of \Cref{thm:minimal-weight-lb} and \Cref{lem:majorization} shows that
\[
 |G| \geq \min\bigl(
 1 + (n-1)(q-1) + (n-1)(q-1)^2,
 1 + n(q-1) + (n-3)(q-1)^2 + (q-1)^3.
 \bigr).
\]
If we subtract the first sum from the second then we obtain
\[
 (q-1)^3 - 2(q-1)^2 + (q-1) = (q-2)^2(q-1) \ge 0,
\]
and so the minimum equals the first sum.
\end{proof}

\subsubsection{Short length}
\label{sec:larger-alphabets-small-n}

In this section we characterize all tight irreducible subcube partitions of $\{0,\dots,q-1\}^n$ for $q \ge 2$ and $n \leq 3$.

It is easy to see that the unique tight irreducible subcube partition of $\{0,\dots,q-1\}^1$ is $\{0,\dots,q-1\}$. In contrast, there is no tight irreducible subcube partition of $\{0,\dots,q-1\}^2$.

\begin{lemma} \label{lem:subcube-partition-2-q}
There are no tight irreducible subcube partitions of $\{0,\dots,q-1\}^2$ for any $q \ge 2$.
\end{lemma}
\begin{proof}
Let $F$ be a tight subcube partition of $\{0,\dots,q-1\}^2$. If all subcubes in $F$ are points then $F$ is clearly reducible. Otherwise, without loss of generality $0* \in F$. For every $a \in \{1,\dots,q-1\}$, let $F_a \subset F$ consist of all subcubes of $F$ starting with $a$. Since $F$ is tight, $F_a \neq \{a*\}$ for some $a$. Since $\bigcup F_a = a*$, it follows that $F$ is reducible.
\end{proof}

Kullmann and Zhao~\cite[Lemma 41]{KullmannZhao16} showed that there is a unique tight irreducible subcube partition of $\{0,1\}^3$, up to flipping coordinates. An analogous result holds for all $q \ge 2$.

\begin{lemma} \label{lem:subcube-partition-3-q}
Every tight irreducible subcube partition of $\{0,\dots,q-1\}^3$, for any $q \ge 2$, can be obtained from $S_3 = \{ 000, 01*, 1*0, *01, 111 \}$ by \Cref{lem:expansion}.
\end{lemma}
\begin{proof}
Let $G$ be a tight irreducible subcube partition of $\{0,\dots,q-1\}^3$. Since $G$ is tight, $*** \notin G$. Furthermore, no subcube in $G$ contains two stars. Indeed, suppose that $0** \in G$. Then for all $a \in \{1,\ldots,q-1\}$, the subcubes in $G$ starting with $a$ together cover $a**$. Since $G$ is irreducible, we see that $G = \bigl\{ a** : a \in \{0,\dots,q-1\} \bigr\}$, contradicting tightness.

Let $A(\pdot ? *)$ denote the projection of all subcubes of $G$ of the form $??*$ to the first coordinate, and define other $A$-sets analogously.

If $A(\pdot ? *) = \emptyset$ then no subcube of $G$ ends with $*$. Therefore the subcubes ending with $b \in \{0,1,2\}$ cover all of $**b$. Since $G$ is irreducible, $**b \in G$, which is impossible. Therefore $A(\pdot ? *) \neq \emptyset$.

We claim that $A(\pdot ? *)$ and $A(\pdot * ?)$ are disjoint. Indeed, if $a \in A(\pdot ? *) \cap A(\pdot * ?)$, then $ab*,a*c \in G$ for some $b,c \in \{0,1,2\}$, which is impossible since these subcubes intersect. 
%Consequently, the $A$-sets are non-empty and different from $\{0,\dots,q-1\}$.

We claim that if $a \in A(\pdot ? *)$ and $b \in A(? \pdot *)$ then $ab* \in G$. Indeed, suppose that $ab'*,a'b* \in G$ but $ab* \notin G$. Consider a point $abc \in \{0,\dots,q-1\}^3$. This point cannot be covered by $a*c$ since this subcube does not conflict with $ab'*$, and cannot be covered by $*bc$ since this subcube does not conflict with $a'b*$. Therefore $abc \in G$. Since this holds for all $c$ and $\bigcup_c abc = ab*$, we get a contradiction with the irreducibility of $G$.

It follows that $G$ is composed of points and edges, where the edges are
\[
 \{ ab* : a \in A(\pdot ? *), b \in A(? \pdot *) \} \cup
 \{ a*c : a \in A(\pdot * ?), c \in A(? * \pdot) \} \cup
 \{ *bc : b \in A(* \pdot ?), c \in A(* ? \pdot) \}.
\]

We claim that $A(\pdot ? *) \cup A(\pdot * ?) = \{0,\dots,q-1\}$, and so these two sets partition $\{0,\dots,q-1\}$. Indeed, suppose that $a$ is contained in neither set. Let $b \in A(? \pdot *)$, so that $b \notin A(* \pdot ?)$. By construction, points of the form $abc$ are not covered by any of the edges of $G$, hence all of them belong to $G$. Since $\bigcup_c abc = ab*$, this contradicts the irreducibility of $G$.

It follows that $G$ can be obtained by applying \Cref{lem:expansion} to $S_3$ with the mappings
\begin{align*}
\phi_1(a) &= 1 \leftrightarrow a \in A(\pdot * ?), &
\phi_2(b) &= 1 \leftrightarrow b \in A(? \pdot *), &
\phi_3(c) &= 1 \leftrightarrow c \in A(* ? \pdot). \end{align*}
Indeed, the edges of $G$ are
\[
 \{ ab* : \phi_1(a) = 0, \phi_2(b) = 1 \} \cup
 \{ a*c : \phi_1(a) = 1, \phi_3(c) = 0 \} \cup
 \{ *bc : \phi_2(b) = 0, \phi_3(c) = 1 \},
\]
and these cover all points $abc \in \{0,\ldots,q-1\}^3$ other than the ones satisfying $\phi_1(a) = \phi_2(b) = \phi_2(c)$.
\end{proof}

\begin{corollary} \label{cor:subcube-partition-3-q}
\Cref{conj:minimal-size-q} holds for $n = 3$ and all $q \ge 2$.
\end{corollary}
\begin{proof}
Let $G$ be a tight irreducible subcube partition of $\{0,\dots,q-1\}^3$, where $q \ge 2$. According to the \namecref{lem:subcube-partition-3-q}, it can be obtained by applying \Cref{lem:expansion}. The result now follows from \Cref{thm:larger-alphabets-expansion-optimal}, since it is known that all tight irreducible subcube partitions of $\{0,1\}^3$ have size $5$.
\end{proof}

When $n \ge 4$, not all tight irreducible subcube partitions are obtained via \Cref{lem:expansion}. Here is an example:
\begin{align*}
& 0000,0002,0020,0022,0101,0102,0111,0122,0200,0201,0211,0220,1010,1011,1020,1021,1102,\\
& 1110,1120,1122,1201,1202,1211,1212,2011,2012,2021,2022,2110,2111,2200,2212,2220,2222,\\ 
& 2100,2101,
01*0,02*2,0*21,10*2,11*1,1*00,20*0,22*1,2*02,
\underline{001*},\underline{122*},\underline{212*},\underline{*001},*\underline{112},\underline{*210}.
\end{align*}

This is a tight irreducible subcube partition of $\{0,1,2\}^4$. The underlined subcubes show that it cannot be obtained by applying \Cref{lem:expansion}, since $001,122,212$ and $001,112,210$ are not product sets.

%However, it is possible that all tight irreducible subcube partitions of $\{0,\dots,q-1\}^n$ of minimal size do arise by applying the \namecref{lem:expansion}.

\subsubsection{Lower bound}
\label{sec:covers}

\Cref{thm:size-lower-bound} gives our best lower bound on the size of a tight irreducible subcube partition of $\{0,1\}^n$, slightly improving on the ``trivial'' lower bound of $n + 1$ which follows from the well-known lemma of Tarsi~\cite{AharoniLinial86} on minimally unsatisfiable CNFs.

Tarsi's lemma applies more generally to subcube covers.

\begin{definition}[Subcube cover] \label{def:subcube-cover}
A \emph{subcube cover} of $\{0,\dots,q-1\}^n$ is a collection of subcubes whose union is $\{0,\dots,q-1\}^n$.

A subcube cover is \emph{minimal} if no proper subset of it is a subcube cover.
\end{definition}

In this language, Tarsi's lemma states that a tight minimal subcube cover of $\{0,1\}^n$ has size at least $n + 1$. This bound is achieved, for example, by the subcube partition
\[
 \{ 0^i 1 *^{n-i-1} : 0 \leq i \leq n-1 \} \cup \{ 0^n \}.
\]
The analogous subcube partition for arbitrary $q \ge 2$ is
\[
 \{ 0^i b *^{n-i-1} : 0 \leq i \leq n-1, 1 \leq b \leq q-1 \} \cup \{ 0^n \},
\]
which has size $(q-1)n + 1$.

In this section, we generalize Tarsi's lemma to the setting of matroids. A special case of our generalization shows that every tight minimal subcube cover of $\{0,\dots,q-1\}^n$ (and so every tight subcube partition of $\{0,\dots,q-1\}^n$) has size at least $(q-1)n + 1$, proving the optimality of the above construction. %Our generalization will also apply in the setting of \Cref{sec:affine-vector-space-partitions}.

%There are several proofs of Tarsi's lemma. The original proof, appearing in a paper of Aharoni and Linial~\cite{AharoniLinial86}, uses Hall's theorem. Other proofs~\cite{MaLiang97,DDKB98,Kullmann00} are inductive. We generalize the original proof.

There are several proofs of Tarsi's lemma~\cite{AharoniLinial86,CSz88,MaLiang97,DDKB98,Kullmann00,BET01}. We generalize the well-known proof using Hall's theorem.

\begin{definition}[Cover] \label{def:hitting}
Let $M$ be a matroid. A collection $F$ of subsets of the ground set of $M$ is an \emph{$M$-cover} if no basis of $M$ intersects all sets in $F$. An $M$-cover is \emph{minimal} if no proper subset is an $M$-cover.
\end{definition}

\begin{theorem}[Generalized Tarsi's lemma] \label{thm:tarsi-general}
Let $M$ be a matroid with rank function $r$. Every minimal $M$-cover $F$ satisfies
\[
 |F| > r\left(\bigcup F\right).
\]
\end{theorem}

The statement might look opaque, so before proving the \namecref{thm:tarsi-general}, we first show how it can be used to derive the lower bound $(q-1)n + 1$.

\begin{theorem} \label{thm:size-lb-q}
Every tight minimal subcube cover of $\{0,\dots,q-1\}^n$, where $n \ge 1$ and $q \ge 2$, has size at least $(q-1)n + 1$.
\end{theorem}
\begin{proof}
Let $H(n,q)$ be the matroid over the ground set $[n] \times \{0,\dots,q-1\}$ in which a set is independent if for every $i \in [n]$, it doesn't contain all elements of the form $(i,?)$. A basis of $H(n,q)$ is any set of the form $B(a_1,\dots,a_n) := \{(i,j) : i \in [n], j \in [q], j \neq a_i\}$, where $a_1,\dots,a_n \in \{0,\dots,q-1\}$.

Let $F$ be a tight minimal subcube cover of $\{0,\dots,q-1\}^n$. We can represent every subcube in $s \in F$ as the following subset of the ground set of $H(n,q)$:
\[
 \phi(s) = \{ (i,s_i) : i \in [n], s_i \neq * \}.
\]

Let $\phi(F) = \{ \phi(s) : s \in F\}$.
We claim that $\phi(F)$ is an $H(n,q)$-cover. Indeed, let $B(a_1,\dots,a_n)$ be any basis of $H(n,q)$. Since $F$ is a subcube cover, the point $a_1\ldots a_n$ is covered by some subcube $s$. If $(i,s_i) \in \phi(s)$ then $s_i = a_i$, and so $\phi(s)$ is disjoint from $B(a_1,\dots,a_n)$.

A similar argument shows that $\phi(F)$ is a minimal $H(n,q)$-cover. Indeed, any proper subset of $\phi(F)$ has the form $\phi(G)$ for some proper subset $G \subset F$. Since $F$ is a minimal subcube cover, some point $a_1\ldots a_n$ is not covered by $G$. The corresponding basis $B(a_1,\dots,a_n)$ intersects all sets in $\phi(G)$. Indeed, if $\phi(s) \in \phi(G)$ then $s$ doesn't cover $a$, and so $s_i \neq a_i,*$ for some $i \in [n]$. Consequently, $\phi(s)$ contains $(i,s_i) \in B(a_1,\dots,a_n)$.

\smallskip

Since $\phi(F)$ is a minimal $H(n,q)$-cover,
\Cref{thm:tarsi-general} shows that $|F| = |\phi(F)|$ exceeds the rank of $\bigcup \phi(s)$. We will show that $\bigcup \phi(s) = [n] \times \{0,\dots,q-1\}$, a set whose rank is $(q - 1)n$, completing the proof.

Let $i \in [n]$. Since $F$ is tight, some subcube $s \in F$ mentions $i$. Since $F$ is minimal, there exists a point $x \in \{0,\dots,q-1\}^n$ which is only covered by $s$. In particular, no subcube of $F$ contains $x^{i \to *}$, the subcube obtained from $x$ by changing the $i$'th coordinate to a star. This implies that for every $b \in \{0,\dots,q-1\}$, every subcube of $F$ containing $x^{i\to b}$ must contain $b$ in its $i$'th coordinate. Therefore $\bigcup \phi(F)$ contains all elements of the form $(i,b)$, for any $b \in \{0,\dots,q-1\}$, as promised.
\end{proof}

The proof of \Cref{thm:tarsi-general} uses a generalization of Hall's theorem to matroids.

\begin{proposition}[{Hall--Rado~\cite{Rado67,Welsh71}}] \label{pro:hall-rado}
Let $M$ be a matroid with rank function $r$, and let $F$ be a collection (multiset) of subsets of the ground set of $M$.

If each subset $G \subseteq F$ satisfies $|G| \leq r(\bigcup G)$ then we can choose an element $e_s$ from each set $s \in F$ such that the elements $e_s$ are distinct, and $\{e_s : s \in F\}$ is an independent set of $M$.
\end{proposition}

We can now prove \Cref{thm:tarsi-general}.

\begin{proof}[Proof of \Cref{thm:tarsi-general}]
Let $F$ be a minimal $M$-cover, and suppose that $|F| \leq r(\bigcup F)$. We will show that this assumption leads to a contradiction.

If every subset $G \subseteq F$ satisfies $|G| \leq r(\bigcup G)$ then \Cref{pro:hall-rado} shows that $F$ intersects the independent set $\{e_s : s \in F\}$. Since every independent set can be completed to a basis, this contradicts the assumption that $F$ is an $M$-cover.

We conclude that some subset $G \subset F$ satisfies $|G| > r(\bigcup G)$. Among all such subsets, choose one which is inclusion-maximal. By assumption, $G \neq F$, and so by minimality, $G$ is not an $M$-cover, say the basis $B$ intersects all sets in $G$.

Let $M' = M/\bigcup G$ be the contraction of $M$ by $\bigcup G$. The ground set of $M'$ is the ground set of $M$ with $\bigcup G$ removed, and its rank function is $r'(S') = r(S' \cup \bigcup G) - r(\bigcup G)$.

Let $F' = \{ S \setminus \bigcup G : S \in F \setminus G \}$. Suppose that $H'$ is a non-empty subset of $F'$, say $H' = \{ S \setminus \bigcup G : S \in H \}$.
Since $G$ is inclusion-maximal,
\[
 r'(H') = r\left(\bigcup H' \cup \bigcup G\right) - r\left(\bigcup G\right) =
 r\left(\bigcup (H \cup G)\right) - r\left(\bigcup G\right)
 > |H \cup G| - |G| = |H| = |H'|.
\]
Applying \Cref{pro:hall-rado}, we obtain a basis $B'$ of $F'$ which intersects all sets in $F'$, and so all sets in $F \setminus G$.

The set $B \cap \bigcup G$ is an independent subset of $\bigcup G$. Complete it to a basis $B_G$ of $M|\bigcup G$. Since $B'$ is a basis of $M'$, $B' \cup B_G$ is a basis of $M$. By construction, $B$ intersects all subsets in $F$, contradicting the assumption that $F$ is an $M$-cover.
\end{proof}

\section{Affine vector space partitions}
\label{sec:affine-vector-space-partitions}

\Cref{sec:subcube-partitions} considers partitions of $\{0,1\}^n$ into subcubes. In this section, we consider partitions of $\{0,1\}^n$ into affine subspaces. The companion work~\cite{BFIK22} considers the more general case of partitions of $\mathbb{F}_q^n$ into affine subspaces.

\begin{definition}[Affine vector space partition] \label{def:avsp}
An \emph{affine vector space partition} of length $n$ is a partition of $\{0,1\}^n$ into \emph{affine subspaces}, that is, sets of the form $x + V$, where $x \in \{0,1\}^n$ and $V$ is a subspace of $\{0,1\}^n$ (identified with $\mathbb{F}_2^n$). The \emph{size} of an affine vector space partition is the number of affine subspaces.

The \emph{linear part} of an affine subspace $U = x + V$ is the subspace $V$. The \emph{dimension} of an affine subspace is the dimension of its linear part, and codimension is defined analogously.
\end{definition}

The notion of reducibility is defined just as in \Cref{def:reducible} and \Cref{def:reducibility-partial}, replacing \emph{subcube} with \emph{affine subspace}.

\begin{definition}[Reducibility] \label{def:avsp-reducibility}
A collection $F$ of disjoint affine subspaces of $\{0,1\}^n$ is \emph{reducible} if there exists a subset $G \subseteq F$, with $|G| > 1$, whose union is an affine subspace of $\{0,1\}^n$ other than $\{0,1\}^n$. If no such $G$ exists then $F$ is \emph{irreducible}.
\end{definition}

The definition of tightness is perhaps less obvious. A subcube partition of length $n$ is not tight if it arises from a subcube partition of length $n-1$ via an embedding of $\{0,1\}^{n-1}$ inside $\{0,1\}^n$. If this is the case, then there is a direction $i$ which is ``ignored'' by all subcubes, in the sense that $s_i = *$. This definition generalizes to our setting, where an affine subspace $x + V$ ``ignores'' a direction $y \in \{0,1\}^n \setminus \{0^n\}$ if $y \in V$.
The same definition was proposed by Agievich~\cite{Agievich08}, under the name \emph{primitivity}, and was dubbed \emph{A-primitivity} by Tarannikov~\cite{Tarannikov22}.

\begin{definition}[Tightness] \label{def:avsp-tightness}
An affine vector space partition $F$ of length $n$ is \emph{tight} if the intersection of the linear parts of all affine subspaces in $F$ is $\{0^n\}$.
\end{definition}

We can determine whether two affine subspaces intersect by solving linear equations. Using this, we can determine whether a collection of affine subspaces forms an affine vector space partition as in \Cref{lem:testing-partition}, by checking that
\[
 \sum_s 2^{-\operatorname{codim}(s)} = 1.
\]

We can check tightness using the definition, and irreducibility using \Cref{alg:reducibility}, suitably generalized. For this we need to be able to compute the join of two affine subspaces, which is the minimal affine subspace containing their union.

\begin{lemma} \label{lem:avsp-join}
The minimal affine subspace containing $a + V$ and $b + W$ is $a + \operatorname{span}(V, W, b-a)$.
\end{lemma}

We leave the straightforward proof to the reader.

\smallskip

We commence the study of affine vector space partitions in \Cref{sec:compression}, where we show how to convert an irreducible subcube partition to an irreducible affine vector space partition. We use this technique in \Cref{sec:avsps-minimal-subcubes} to construct tight irreducible affine vector space partitions of length $n$ and size $\tfrac32n - O(1)$.
In the same section we also prove a lower bound of $n+1$ on the size of a tight irreducible affine vector space partition of length $n$.

\subsection{Compression}
\label{sec:compression}

Every subcube partition of length $n$ can be viewed as an affine vector space partition of length $n$. Furthermore, if the subcube partition is tight, then so is the affine vector space partition. However, irreducibility is not maintained in this conversion. For example,
\[
 000, 111, 01*, 1*0, *01
\]
is irreducible as a subcube partition but reducible as an affine vector space partition, since $000 \cup 111$ is an affine subspace, which we can represent by $aaa$. If we merge these two points, we get the tight irreducible affine vector space partition
\[
 aaa, 01*, 1*0, *01.
\]

In this section we generalize this process of merging for arbitrary irreducible subcube partitions, using the concept of \emph{star pattern} introduced in \Cref{def:star-pattern}: the star pattern of a subcube $s \in \{0,1,*\}^n$ is $P(s) := \{ i \in [n] : s_i = * \}$.

\begin{lemma} \label{lem:compression}
Let $F$ be an irreducible subcube partition. For $S \subseteq [n]$, let $F_S$ consist of all subcubes in $F$ whose star pattern is $S$. For each $S \subseteq [n]$, choose a partition of $F_S$ in which the union of each part is an affine subspace, and let $G_S$ be the corresponding collection of affine subspaces. (If $F_S = \emptyset$, take $G_S = \emptyset$.)

If all $G_S$ are irreducible then $G = \bigcup_S G_S$ is also irreducible.
\end{lemma}
\begin{proof}
If $G$ is reducible then there exists a subset $G' \subset G$, with $|G'| > 1$, whose union is an affine subspace $U$ other than $\{0,1\}^n$.
Each affine subspace in $G'$ is a union of subcubes of $F$. Let $F'$ be the collection of all such subcubes, so that $\bigcup F' = U$.

If $F'$ contains a subcube $s$ with $s_i = *$ then according to \Cref{lem:avsp-join}, the linear part of $U$ contains $0^{i-1}10^{n-i}$. This motivates defining $S$ as the set of coordinates $i \in [n]$ such that $s_i = *$ for some $s \in F'$. Note that $S \neq [n]$, since otherwise $U = \{0,1\}^n$.

Let $U|_{\bar{S}}$ be the projection of $U$ into the coordinates outside of $S$, so that
\[
 U = \{ x \in \{0,1\}^n : x|_{\bar{S}} \in U|_{\bar{S}} \}.
\]
Let $y \in U|_{\bar{S}}$. Every $x \in \{0,1\}^n$ such that $x|_{\bar{S}} = y$ is covered by some subcube $s \in F'$. The definition of $S$ implies that $s_i = y_i$ for all $i \in \bar{S}$. Consequently the union of all subcubes $s \in F'$ such that $s|_{\bar{S}} = y$ is the subcube $s_y := \{ x \in \{0,1\}^n : x|_{\bar{S}} = y \}$.
Since $F$ is irreducible, $s_y \in F$ and so $s_y \in F'$. Since $s_y \in F_S$, it follows that $F' \subseteq F_S$, and so $G' \subseteq G_S$. This contradicts the irreducibility of $G'$.
\end{proof}

In general, $F$ being tight doesn't guarantee that $G$ is tight. For example, applying \Cref{lem:compression} to the tight subcube partition
\[
 *000, *111, 001*, 0*01, 01*0, 110*, 1*10, 10*1
\]
results in the non-tight affine vector space partition
\[
 *aaa, aa\bar{a}*, a*a\bar{a}, a\bar{a}*a,
\]
in which all linear parts contain the non-zero vector $1111$.

\smallskip

The following \namecref{lem:compression-easy} is a simplification of \Cref{lem:compression} which also includes a criterion for tightness.

\begin{lemma} \label{lem:compression-easy}
Let $F$ be an irreducible subcube partition. For $S \subseteq [n]$, let $F_S$ consist of all subcubes in $F$ whose star pattern is $S$.

Suppose that whenever $F_S$ is non-empty, the union of all subcubes in $F_S$ is an affine subspace $g_S$ (this is always the case when $|F_S| \le 2$). Then $G = \{ g_S : F_S \neq \emptyset \}$ is an irreducible affine vector space partition.

Furthermore, $G$ is tight if
\begin{equation} \label{eq:tightness-condition}
 \bigcap_{S\colon F_S \neq \emptyset} P\left(\bigvee F_S\right) = \emptyset,
\end{equation}
where the join is taken in the sense of subcubes.
\end{lemma}
\begin{proof}
The irreducibility of $G$ follows directly from \Cref{lem:compression}. Indeed, if for every non-empty $F_S$ we take the partition consisting of a single part then the affine vector space partition $G$ in this \namecref{lem:compression-easy} coincides with that in \Cref{lem:compression}. Moreover, if $F_S = \{ a \}$ then $a$ is itself an affine subspace, and if $F_S = \{ a, b \}$ then $a \cup b$ is the affine subspace obtained from $a$ by adding the following vector to the linear part: $v_i = 1$ if $a_i \neq b_i$ and $v_i = 0$ otherwise.

We proceed to show that if \Cref{eq:tightness-condition} holds then $G$ is tight.
Let $S \subseteq [n]$ be such that $F_S$ is non-empty. If $i \notin P(\bigvee F_S)$ then all $x \in g_S$ have the same value of $x_i$, and so $y_i = 0$ for all $y$ in the linear part of $g_S$. Therefore if $y_i = 1$ for some $y$ in the linear part of $g_S$ then $i \in P(\bigvee F_S)$.
\Cref{eq:tightness-condition}
thus guarantees that the only vector in the intersection of the linear parts of all $g_S$ is the zero vector, and so $G$ is tight.
\end{proof}

\subsection{Minimal size}
\label{sec:avsps-minimal-subcubes}

In \Cref{sec:minimal-subcubes} we conjectured that the minimal size of a tight irreducible subcube partition of length $n$ is $2n-1$. Using a computer, we have determined the minimal size of a tight irreducible affine vector space partition of length $n$ for small $n$~\cite{BFIK22}:
\[
 \begin{array}{r|ccccc}
 n & 3 & 4 & 5 & 6 & 7 \\\hline
 \text{scp} & 5 & 7 & 9 & 11 & 13 \\
 \text{avsp} & 4 & 6 & 7 & 8 & 10
 \end{array}
\]
The first row is the minimal size of a tight irreducible subcube partition of length $n$, and the second row is the minimal size of a tight irreducible affine vector space partition of length $n$.

The constructions presented later in this section suggest the following conjecture.

\begin{conjecture} \label{conj:avsp-min-size}
The minimal size of a tight irreducible affine vector space partition is $\frac{3}{2} n - o(n)$.
\end{conjecture}

We give a matching construction in \Cref{sec:avsp-construction}. The best lower bound we are aware of is $n + 1$, which we prove in \Cref{sec:avsp-lb} using an argument similar to the proof of \Cref{thm:tarsi-general}.

\subsubsection{Lower bound} \label{sec:avsp-lb}

In this section, we adapt the proof of \Cref{thm:tarsi-general} to the setting of affine vector space partitions.

\begin{theorem} \label{thm:avsp-lb}
Every tight affine vector space partition of length $n \ge 1$ has size at least $n + 1$.
\end{theorem}

As in \Cref{thm:size-lb-q}, the lower bound holds more generally for every tight minimal affine vector space cover, a concept we do not define formally.

\begin{proof}

\newcommand{\varF}{\mathbf F}
\newcommand{\varG}{\mathbf G}
\newcommand{\VG}{\operatorname{span}(\varG)}

Let $M$ be the matroid over $\{0,1\}^n$ in which a subset is independent if it is linearly independent, and let $r$ be its rank function.

Suppose that $F$ a tight irreducible affine vector space partition of length $n \ge 1$, and
let $\varF = \{ V^\perp : x + V \in F \}$ (which we consider as a multiset).

We claim that $r(\varF) = n$. Indeed, since $F$ is tight,
\[
 \operatorname{span}(\{V^\perp : x + V \in F\}) = 
 \left(\bigcap \{ V : x + V \in F \}\right)^\perp = \{0^n\}^\perp = \{0,1\}^n,
\]
and so $r(\varF) = n$.

Suppose that every $\varG \subseteq \varF$ satisfies $|\varG| \leq r(\bigcup \varG)$. According to \Cref{pro:hall-rado}, we can choose an element $y_{x+V} \in V^\perp$ for each $x+V \in F$ such that the elements $y_{x+V}$ form an independent set. In particular, we can find an element $z \in \{0,1\}^n$ such that $\langle z,y_{x+V} \rangle \neq \langle x,y_{x+V} \rangle$ for all $x + V \in F$. By construction, $z \notin x + V$ for all $x + V \in F$, contradicting the assumption that $F$ is an affine vector space partition.

It follows that there exists some subset $\varG \subseteq \varF$ satisfying $|\varG| > r(\bigcup \varG)$. Among all such subsets, choose one which is inclusion-maximal. If $\varG = \varF$ then we are done, so suppose that $\varG \neq \varF$. Since $\varF$ is an affine vector space partition, there is a point $z_0$ which is not covered by any subspace in $\varG$. 

%Let $\VG = \operatorname{span}(\{V^\perp : V^\perp \in \varG\})$.
Let $M' = M/\VG$, and let $r'$ be its rank function. Let $\varF' = \{ V^\perp \setminus \VG : V^\perp \in \varF \setminus \varG \}$. If $\varG' \subseteq \varF'$ then
\[
 r'(\varG') = r(\varG' \cup \VG) - r(\VG) = r(\varG' \cup \varG) - r(\varG) > |\varG' \cup \varG| - |\varG| = |\varG'|,
\]
using the inclusion-maximality of $\varG'$. Hence \Cref{pro:hall-rado} allows us to choose $y'_{x+V} \in V^\perp \setminus \VG$ for all $x + V \in F \setminus G$ such that these vectors are independent in $M'$, which means that no linear combination of them lies in $\VG$ (and in particular, they are linearly independent).

Let $z$ be a point such that $\langle z, y'_{x+V} \rangle \neq \langle x, y'_{x+V} \rangle$ for all $x+V \in F \setminus G$ and $\langle z,y \rangle = \langle z_0,y \rangle$ for all $y \in \VG$. %; such a point exists since no linear combination of the vectors $y'_{x+V}$ lies in $\VG$.
By construction, $z$ is not covered by any of the subspaces of $F \setminus G$. It is also not contained in any $x + V \in G$, since $x + V = \{ w : \langle w,y \rangle = \langle x,y \rangle \text{ for all } y \in V^\perp\}$ and $\langle z,y \rangle = \langle z_0,y \rangle$ for all $y \in V^\perp$ (recalling that $z_0$ is not covered by $G$). This contradicts the assumption that $F$ is an affine vector space partition.
\end{proof}

% \begin{proof}
% Let $M$ be the matroid over the ground set $\{0,1\}^{n+1}$ in which a subset is independent if its projection to the first $n$ coordinates is linearly independent.

% For each affine subspace $s \in F$, let $\phi(s)$ be the collection of all pairs $(y,b)$ such that $y \in \{0,1\}^n$, $b \in \{0,1\}$, and $\langle x,y \rangle \neq b$ for all $x \in s$.

% We claim that $\phi(F) = \{ \phi(s) : s \in F \}$ is an $M$-cover. Suppose that some basis $B$ intersects all affine subspaces in $F$. Since $B$ is a basis, we can find $x \in \{0,1\}^n$ such that $\langle x,y \rangle = b$ for all $(y,b) \in B$. Let $s \in F$ be the affine subspace containing $x$. If $(y,b) \in B \cap \phi(s)$ then $\langle x,y \rangle \neq b$ by definition of $\phi(s)$, contradicting the definition of $x$.

% Applying the argument in the proof of \Cref{thm:tarsi-general}, we find a proper subset $G \subset F$ and a basis $B'$ of $M' = M/\bigcup G$ which intersects all sets in $F' = \{ S \setminus \bigcup G : S \in F \setminus G \}$. %Since every element in $V_G = \operatornam{span}(\bigcup G)$ is a loop of $M'$ (this follows from the formula for its rank function), the basis $B'$ is disjoint from $V_G$.

% %Since $G$ is a proper subset of $F$, there is a point $x$ not covered by the affine subspaces in $G$. Any point $x'$ such that $\langle x',y \rangle = \langle x,y \rangle$ for all $y$ such that $?$
% \end{proof}

It is tempting to conjecture a common generalization of \Cref{thm:size-lb-q} and \Cref{thm:avsp-lb}, namely that a tight affine vector space partition of $\mathbb{F}_q^n$ has size at least $(q-1)n + 1$. 
Unlike \Cref{thm:size-lb-q} and \Cref{thm:avsp-lb}, this cannot be 
true for tight minimal affine vector space covers for $q \geq 4$ (hence,
the proof above cannot generalize): 
we can construct tight minimal affine vector space covers of size $(q-1)(n-3) + \frac32 (q+1)-1$ for $q$ odd,
and we can construct tight minimal affine vector space covers of 
size $(q-1)(n-3) + q+q/p$ for $q=p^h$, where $p$ is a prime.
These constructions derive 
from the two examples of minimal blocking sets in a projective 
plane described in \cite{BB1986}. We leave the details
for elsewhere.

\subsubsection{Construction} \label{sec:avsp-construction}

In this section, we construct tight irreducible affine vector space partitions of length $n$ and size $\tfrac32 n - O(1)$ for all $n \ge 3$, using \Cref{lem:compression-easy}. To construct the underlying subcube partitions, we use an inductive approach in the style of the constructions in \Cref{sec:minimal-subcubes,sec:minimal-weight}.

\begin{lemma} \label{lem:avsp-inductive}
Let $F,H$ be irreducible subcube partitions of length $n \ge 2$ satisfying the following conditions, where $F_S$ consists of all subcubes in $F$ whose star pattern in $S$:
\begin{enumerate}[(i)]
\item $\{ s \in F : s_1 = * \} = \{ s \in H : s_1 = * \}$.
\item $|F_S|,|H_S| \leq 2$ for all $S$.
\item \Cref{eq:tightness-condition} holds for $H$.
\end{enumerate}

Let $m = m(H)$ be the number of star patterns $S$ such that $H_S$ is non-empty, and let $m' = m'(H)$ be the number of those star patterns where $1 \notin S$ (that is, the first coordinate is not a star).

% Let $F$ be an irreducible subcube partition of length $n \ge 2$ satisfying the following conditions, where $F_S$ consists of all subcubes in $F$ whose star pattern is $S$:
% \begin{enumerate}[(i)]
% \item $|F_S| \leq 2$ for all $S$.
% \item \Cref{eq:tightness-condition} holds.
% \end{enumerate}
% For $p \in \{1,2\}$, let $m_p = m_p(F)$ be the number of star patterns $S$ such that $|F_S| = p$, and let $m'_p$ the number of those star patterns where $1 \notin S$.

For every $k \ge 0$ there exists a tight irreducible affine vector space partition of length $n + 2k$ and size $m + km'$.
\end{lemma}
\begin{proof}
For $N \geq 0$, let $F*^N = \{ s*^N : s \in F \}$, and define $*^NF$ similarly.

Let $F^* = \{ s \in \{0,1\}^{n-1} : *s \in F \}$. If $F^*$ is empty then the union of the subcubes starting with $b \in \{0,1\}$ is $b*^{n-1}$. Since $F$ is irreducible, $F = \{0*^{n-1}, 1*^{n-1}\}$. However, this contradicts \Cref{eq:tightness-condition}, using $n \ge 2$. Therefore $F^*$ is non-empty.

We will construct an infinite sequence of subcube partitions $F^{(k)}$ such that the following hold:

\begin{enumerate}[(i)]
\item $F^{(k)}$ is an irreducible subcube partition of length $n + 2k$.
\item $F^**^{2k+1} \subseteq F^{(k)}$.
\item Subcubes in $F^**^{2k+1}$ have different star patterns from subcubes in $F^{(k)} \setminus F^**^{2k+1}$.
\item $|F^{(k)}_S| \leq 2$ for all $S$.
\item $m(F^{(k)}) = m + k m'$.
\item \Cref{eq:tightness-condition} holds for $F^{(k)}$.
\end{enumerate}

\noindent The result then follows by applying \Cref{lem:compression-easy} to $F^{(k)}$.

\smallskip

The starting point is $F^{(0)} = \{ sa : as \in H, |a| = 1, |s|=n-1 \}$.
By assumption, $F^{(0)}$ is irreducible and $F^** = H^**$. By construction, $F^** \subseteq F^{(0)}$, and all subcubes in $F^{(0)} \setminus F^**$ end with a non-star. The remaining properties are by assumption.

Given $F^{(k)}$, we construct $F^{(k+1)}$ as follows. Apply \Cref{lem:merge} with $F_0 = *F^{(k)}$ and $F_1 = F*^{2k+1}$ to obtain a subcube partition $G^{(k)}$. We define $F^{(k+1)} = \{tab : abt \in G^{(k)}, |a|=|b|=1, |t|=n+2k \}$. Since $F^**^{2k+1} \subseteq F^{(k)}$, we can explicitly write
\[
 F^{(k+1)} =
 \{ t** : t \in F^**^{2k+1} \} \cup
 \{ t0* : t \in F^{(k)} \setminus F^**^{2k+1} \} \cup
 \{ t1b : bt \in F*^{2k+1}, b \neq * \}.
\]

We now verify the properties of $F^{(k+1)}$ one by one:
\begin{enumerate}[(i)]
\item By the induction hypothesis, $F^**^{2k+1} \subseteq F^{(k)}$, and so $*F^**^{2k+1} \subseteq *F^{(k)}$. Since $F^*$ is non-empty, $F_0 \cap F_1 = *F^**^{2k+1}$ is non-empty. Therefore \Cref{lem:merge} shows that $G^{(k)}$ is irreducible, and it follows that $F^{(k+1)}$ is irreducible.
\item The formula for $F^{(k+1)}$ immediately implies that $F^**^{2k+3} \subseteq F^{(k+1)}$.
\item The formula for $F^{(k+1)}$ shows that all $s \in F^{(k+1)} \setminus F^**^{2k+3}$ satisfy $s_{n+2k+1} \neq *$, and so have different star patterns from any subcube in $F^**^{2k+3}$.
\item The subcubes in each of the three sets in the formula for $F^{(k+1)}$ have different star patterns. Since $|F^*_S| \leq 2$ for all $S$ and $|F^{(k)}_S| \leq 2$ for all $S$, it follows that $|F^{(k+1)}_S| \leq 2$ for all $S$.
\item Denote the three parts in the formula for $F^{(k+1)}$ by $A,B,C$. Clearly $m(A) = m(F^*)$. Since the star patterns of the subcubes in $F^**^{2k+1}$ are different from the star patterns of the subcubes in $F^{(k)} \setminus F^**^{2k+1}$, we have $m(B) = m(F^{(k)}) - m(F^*)$. Finally, $m(C) = m'$. We conclude that $m(F^{(k+1)}) = m(F^{(k)}) + m' = m + (k + 1) m'$.
\item Since the star patterns of the subcubes in $F^**^{2k+1}$ are different from the star patterns of the subcubes in $F^{(k)} \setminus F^**^{2k+1}$, the induction hypothesis implies that the intersection of $P(\bigvee F_S)$ for all star patterns $S$ appearing in $A \cup B$ is contained in $\{n+2k+1,n+2k+2\}$.

\Cref{eq:tightness-condition} for $F$ implies that some $s \in F$ satisfies $s_1 \neq *$, and so $C$ is non-empty. All star patterns of subcubes in $S$ do not contain $n+2k+1$ or $n+2k+2$, and so \Cref{eq:tightness-condition} holds for $F^{(k+1)}$. \qedhere
\end{enumerate}
\end{proof}

Using this \namecref{lem:avsp-inductive}, we construct tight irreducible affine vector space partitions of length $n$ and size $\tfrac32 n - O(1)$ for all $n \ge 3$. Our construction matches the optimal values in the table appearing in the beginning of the section.

\begin{theorem} \label{thm:avsp-construction}
For all odd $n \ge 3$ there is a tight irreducible affine vector space partition of length $n$ and size $\frac32n-\frac12$.

There is a tight irreducible affine vector space partition of length $4$ and size $6$.

For all even $n \ge 6$ there is a tight irreducible affine vector space partition of length $n$ and size $\frac32n - 1$.
\end{theorem}
\begin{proof}
Consider the following subcube partitions:
\begin{align*}
S_3 &= \{*01, 000, 111, 1*0, 01*\}, \\
S_4 &= \{*01*, 1000, 1111, 11*0, 1*01, 000*, 01**\}, \\
T_6 &= \{*0110*, *1101*, *001*1, *010*0, *00**0, *0*0*1, **111*, 011*0*, 110*1*, 010***, 11**0*\}.
\end{align*}
% Consider the following subcube partitions $S_3,S_4,T_6$:
% \begin{align*}
% &*01 & &*01* & &*0110* \\
% &000 & &1000 & &*1101* \\
% &111 & &1111 & &*001*1 \\
% &1*0 & &11*0 & &*010*0 \\
% &01* & &1*01 & &*00**0 \\
% &    & &000* & &*0*0*1 \\
% &    & &01** & &**111* \\
% &    & &     & &011*0* \\
% &    & &     & &110*1* \\
% &    & &     & &010*** \\
% &    & &     & &11**0* \\
% \end{align*}

Using \Cref{alg:reducibility}, one can check that they are irreducible. One checks directly that the prerequisites of \Cref{lem:avsp-inductive} hold in all cases (with $H = F$).

Since $m(S_3) = 4$ and $m'(S_3) = 3$, \Cref{lem:avsp-inductive} with $F = H = S_3$ constructs tight irreducible affine subspace partitions of length $3 + 2k$ and size $4 + 3k = \tfrac{3}{2} (3 + 2k) - \tfrac12$.

Applying \Cref{lem:compression-easy} directly to $S_4$, we obtain a tight irreducible affine subspace partition of length $4$ and size $6$.

Since $m(T_6) = 8$ and $m'(T_6) = 3$, \Cref{lem:avsp-inductive} with $F = H = T_6$ constructs tight irreducible affine subspace partitions of length $6 + 2k$ and size $8 + 3k = \tfrac{3}{2} (6 + 2k) - 1$.
\end{proof}

Applying \Cref{lem:avsp-inductive} with $F = S_3*$ and $H = S_4$ constructs tight irreducible affine subspace partitions of length $4 + 2k$ and size $6 + 3k = \frac{3}{2}(4 + 2k)$, which is slightly worse than what we get using $F = H = T_6$.

\bibliographystyle{alpha}
\bibliography{ascp}

\begin{thebibliography}{DDKB98}

\bibitem[Agi08]{Agievich08}
S.~V. Agievich.
\newblock Bent rectangles.
\newblock In {\em Boolean Functions in Cryptology and Information Security},
  volume~18 of {\em NATO Sci. Peace Secur. Ser. D: Inf. Commun. Secur.}, 2008.

\bibitem[AL86]{AharoniLinial86}
Ron Aharoni and Nathan Linial.
\newblock Minimal non-two-colorable hypergraphs and minimal unsatisfiable
  formulas.
\newblock {\em J. Combin. Theory Ser. A}, 43(2):196--204, 1986.

\bibitem[BB86]{BB1986}
A.~Blokhuis and A.~E. Brouwer.
\newblock Blocking sets in desarguesian projective planes.
\newblock {\em Bulletin of the London Mathematical Society}, 18(2):132--134,
  1986.

\bibitem[BET01]{BET01}
Sven Baumer, Juan~Luis Esteban, and Jacobo Tor\'{a}n.
\newblock Minimally unsatisfiable {CNF} formulas.
\newblock {\em Bull. Eur. Assoc. Theor. Comput. Sci. EATCS}, 74:190--192, 2001.

\bibitem[BFF90]{BFF90}
Marc~A. Berger, Alexander Felzenbaum, and Aviezri~S. Fraenkel.
\newblock Irreducible disjoint covering systems (with an application to
  {B}oolean algebra).
\newblock {\em Discrete Appl. Math.}, 29(2-3):143--164, 1990.
\newblock First International Colloquium on Pseudo-Boolean Optimization and
  Related Topics (Chexbres, 1987).

\bibitem[BFIK23]{BFIK22}
John Bamberg, Yuval Filmus, Ferdinand Ihringer, and Sascha Kurz.
\newblock Affine vector space partitions.
\newblock {\em Des. Codes Cryptogr.}, 2023.

\bibitem[BOH90]{BOH90}
Y.~Brandman, A.~Orlitsky, and J.~Hennessy.
\newblock A spectral lower bound technique for the size of decision trees and
  two-level {AND}/{OR} circuits.
\newblock {\em IEEE Trans. Comput.}, 39(2):282–287, feb 1990.

\bibitem[CH11]{CH11}
Yves Crama and Peter~L. Hammer.
\newblock {\em Boolean functions}, volume 142 of {\em Encyclopedia of
  Mathematics and its Applications}.
\newblock Cambridge University Press, Cambridge, 2011.
\newblock Theory, algorithms, and applications.

\bibitem[CS88]{CSz88}
Va\v{s}ek Chv\'{a}tal and Endre Szemer\'{e}di.
\newblock Many hard examples for resolution.
\newblock {\em J. Assoc. Comput. Mach.}, 35(4):759--768, 1988.

\bibitem[DD98]{DavDav98}
G.~Davydov and I.~Davydova.
\newblock Dividing formulas and polynomial classes for satisfiability.
\newblock In {\em SAT’98, 2nd Workshop on the Satisfiability Problem}, page
  12–21, 1998.

\bibitem[DDKB98]{DDKB98}
Gennady Davydov, Inna Davydova, and Hans Kleine~B\"{u}ning.
\newblock An efficient algorithm for the minimal unsatisfiability problem for a
  subclass of {CNF}.
\newblock {\em Ann. Math. Artificial Intelligence}, 23(3-4):229--245, 1998.

\bibitem[FKW02]{FKW02}
Ehud Friedgut, Jeff Kahn, and Avi Wigderson.
\newblock Computing graph properties by randomized subcube partitions.
\newblock In Jos{\'e} D.~P. Rolim and Salil Vadhan, editors, {\em Randomization
  and Approximation Techniques in Computer Science}, pages 105--113, Berlin,
  Heidelberg, 2002. Springer Berlin Heidelberg.

\bibitem[For73]{Forcade73}
Rodney Forcade.
\newblock Smallest maximal matchings in the graph of the {$d$}-dimensional
  cube.
\newblock {\em J. Combinatorial Theory Ser. B}, 14:153--156, 1973.

\bibitem[GK13]{GwynneKullman13}
Matthew Gwynne and Oliver Kullmann.
\newblock Towards a theory of good {SAT} representations.
\newblock {\em CoRR}, abs/1302.4421, 2013.

\bibitem[GPW18]{GPW18}
Mika G\"{o}\"{o}s, Toniann Pitassi, and Thomas Watson.
\newblock Deterministic communication vs. partition number.
\newblock {\em SIAM J. Comput.}, 47(6):2435--2450, 2018.

\bibitem[Iwa87]{Iwama87}
Kazuo Iwama.
\newblock Complementary approaches to {CNF} {B}oolean equations.
\newblock In {\em Discrete algorithms and complexity ({K}yoto, 1986)},
  volume~15 of {\em Perspect. Comput.}, pages 223--236. Academic Press, Boston,
  MA, 1987.

\bibitem[Iwa89]{Iwama89}
Kazuo Iwama.
\newblock C{NF}-satisfiability test by counting and polynomial average time.
\newblock {\em SIAM J. Comput.}, 18(2):385--391, 1989.

\bibitem[KB00]{KB00}
Hans Kleine~B\"{u}ning.
\newblock On subclasses of minimal unsatisfiable formulas.
\newblock {\em Discrete Appl. Math.}, 107(1-3):83--98, 2000.
\newblock Boolean functions and related problems.

\bibitem[Kis14]{Kisielewicz14}
Andrzej~P. Kisielewicz.
\newblock Partitions and balanced matchings of an {$n$}-dimensional cube.
\newblock {\em European J. Combin.}, 40:93--107, 2014.

\bibitem[Kis20]{Kis20}
Andrzej~P. Kisielewicz.
\newblock On the structure of cube tiling codes.
\newblock {\em European Journal of Combinatorics}, 89:103168, 2020.

\bibitem[Kis23]{Kisielewicz23}
Andrzej~P. Kisielewicz.
\newblock Private communication, 2023.

\bibitem[Kor84]{Korec84}
Ivan Korec.
\newblock Irreducible disjoint covering systems.
\newblock {\em Acta Arith.}, 44(4):389--395, 1984.

\bibitem[KP08]{KP08}
Andrzej~P. Kisielewicz and Krzysztof Przes{\l}awski.
\newblock Polyboxes, cube tilings and rigidity.
\newblock {\em Discrete Comput. Geom.}, 40(1):1--30, 2008.

\bibitem[Kul00]{Kullmann00}
Oliver Kullmann.
\newblock An application of matroid theory to the {SAT} problem.
\newblock In {\em 15th {A}nnual {IEEE} {C}onference on {C}omputational
  {C}omplexity ({F}lorence, 2000)}, pages 116--124. IEEE Computer Soc., Los
  Alamitos, CA, 2000.

\bibitem[Kul04]{Kullmann04}
Oliver Kullmann.
\newblock The combinatorics of conflicts between clauses.
\newblock In Enrico Giunchiglia and Armando Tacchella, editors, {\em Theory and
  Applications of Satisfiability Testing}, pages 426--440, Berlin, Heidelberg,
  2004. Springer Berlin Heidelberg.

\bibitem[KZ13]{KullmannZhao13}
Oliver Kullmann and Xishun Zhao.
\newblock On {D}avis-{P}utnam reductions for minimally unsatisfiable
  clause-sets.
\newblock {\em Theoret. Comput. Sci.}, 492:70--87, 2013.

\bibitem[KZ16]{KullmannZhao16}
Oliver Kullmann and Xishun Zhao.
\newblock Unsatisfiable hitting clause-sets with three more clauses than
  variables.
\newblock {\em CoRR}, abs/1604.01288, 2016.

\bibitem[LS94]{LagariasShor94}
J.~C. Lagarias and P.~W. Shor.
\newblock Cube-tilings of {${\bf R}^n$} and nonlinear codes.
\newblock {\em Discrete Comput. Geom.}, 11(4):359--391, 1994.

\bibitem[ML97]{MaLiang97}
Shaohan Ma and Dongmin Liang.
\newblock A polynomial-time algorithm for reducing the number of variables in
  {MAX} {SAT} problem.
\newblock {\em Sci. China Ser. E}, 40(3):301--311, 1997.

\bibitem[Per05]{Perezhogin05}
A.~L. Perezhogin.
\newblock О специальных совершенных
  паросочетаниях в булевом кубе ({E}ngl.: {O}n
  special perfect matchings in a {B}oolean cube).
\newblock {\em Diskretn. Anal. Issled. Oper. Ser. 1}, 12(4):51--59, 2005.

\bibitem[PS23]{PeitlSzeider22}
Tomáš Peitl and Stefan Szeider.
\newblock Are hitting formulas hard for resolution?
\newblock {\em Discrete Applied Mathematics}, 337:173--184, 2023.

\bibitem[Rad67]{Rado67}
R.~Rado.
\newblock Note on the transfinite case of {H}all's theorem on representatives.
\newblock {\em J. London Math. Soc.}, 42:321--324, 1967.

\bibitem[Tar22]{Tarannikov22}
Yu.~V. Tarannikov.
\newblock О существовании разбиений,
  примитивных по Агиевичу ({E}ngl.: {O}n the existence of
  {A}gievich-primitive partitions).
\newblock {\em Diskretn. Anal. Issled. Oper. Ser. 1}, 29(4):104--123, 2022.

\bibitem[Tar23]{Tarannikov23}
Yu.~V. Tarannikov.
\newblock Private communication, 2023.

\bibitem[Wel71]{Welsh71}
D.~J.~A. Welsh.
\newblock Generalized versions of {H}all's theorem.
\newblock {\em J. Combinatorial Theory Ser. B}, 10:95--101, 1971.

\end{thebibliography}

\end{document}